\documentclass{amsart}

\usepackage{centernot}
\usepackage{dsfont}
\usepackage{amsmath}
\usepackage{leftidx}
\usepackage{amsthm}
\usepackage{amsfonts}
\usepackage{mathtools}
\usepackage{color}
\usepackage[margin=42mm]{geometry}

\newcommand{\reals}{\mathbb{R}}

\newcommand{\complex}{\mathbb{C}}

\newcommand{\otimesc}{\otimes_{\complex}}

\newcommand{\paraa}[1]{\big(#1\big)}

\newcommand{\End}{\operatorname{End}}
\newcommand{\Aut}{\operatorname{Aut}}
\newcommand{\Id}{\operatorname{Id}}
\newcommand{\Der}{\operatorname{Der}}

\newcommand{\spacearound}[1]{\quad#1\quad}
\newcommand{\equivalent}{\spacearound{\Leftrightarrow}}
\renewcommand{\implies}{\spacearound{\Rightarrow}}

\newtheorem{theorem}{Theorem}[section]
\newtheorem{corollary}[theorem]{Corollary}
\newtheorem{lemma}[theorem]{Lemma}
\newtheorem{proposition}[theorem]{Proposition}
\newtheorem{example}[theorem]{Example}
\theoremstyle{definition}
\newtheorem{definition}[theorem]{Definition}
\theoremstyle{remark}

\numberwithin{equation}{section}

\renewcommand{\mid}{\mathds{1}}

\newcommand{\A}{\mathcal{A}}

\newcommand{\Gammat}{\tilde{\Gamma}}
\newcommand{\sigmatau}{(\sigma,\tau)}
\newcommand{\tausigma}{(\tau,\sigma)}

\newcommand{\sigmah}{\hat{\sigma}}
\newcommand{\tauh}{\hat{\tau}}

\newcommand{\TSigma}{T\Sigma}
\newcommand{\mat}{\mbox{Mat}_{N}(\mathbb{C})}
\newcommand{\matp}{\mbox{Mat}_{N}(\mathbb{C})p}

\newcommand{\curv}{\mbox{Curv}}
\newcommand{\deltah}{\hat{\delta}}
\newcommand{\Msigmatau}{M_{(\sigma, \tau)}}

\newcommand{\K}{\mathbb{K}}
\newcommand{\EndA}{\End(\A)}
\renewcommand{\Im}{\operatorname{Im}}

\title[Symmetric $(\sigma,\tau)$-algebras and $\sigmatau$-Hochschild cohomology]{Symmetric $(\sigma,\tau)$-algebras and\\ $\sigmatau$-Hochschild cohomology}
\author{Kwalombota Ilwale}

\address[Kwalombota Ilwale]{Dept. of Math.\\
	Link\"oping University\\
	581 83 Link\"oping\\
	Sweden}
\email{kwalombota.ilwale@liu.se}

\subjclass[2000]{}

\begin{document}
	
\maketitle

\begin{abstract}
  On an associative algebra, we introduce the concept of symmetric
  $\sigmatau$-derivations together with a regularity condition and
  prove that strongly regular symmetric $\sigmatau$-derivations are
  inner. Symmetric $\sigmatau$-derivations are $\sigmatau$-derivations
  that are simultaneously $\sigmatau$-derivations as well as
  $\tausigma$-derivations, generalizing a property of commutative
  algebras.  Motivated by this notion, we explore the geometry of
  symmetric $\sigmatau$-algebras and prove that there exist a unique
  strongly regular symmetric $\sigmatau$-connection. Furthermore, we
  introduce $\sigmatau$-Hochschild cohomology and show that, in first
  degree, it describes the outer $\sigmatau$-derivations on an
  associative algebra. Along the way, examples are provided to
  illustrate the novel concepts.
\end{abstract}

\section{Introduction}

\noindent
In a derivation-based approach to noncommutative geometry, derivations
play a significant role, especially in describing connections,
curvature, and the torsion of a connection. Let $\A$ be a unital
  associative algebra over $\complex$. For a
  derivation $X$, i.e a $\complex$-linear map $X:\A\to \A$ satisfying
  Leibniz rule $X(fg) = fX(g) + X(f)g$, a connection
$\nabla_{X}: M\to M,$ on a left $\A$-module $M,$ satisfies Leibniz
rule
\begin{eqnarray*}
	\nabla_{X}(fm) = f\nabla_{X}(m) + X(f)m
\end{eqnarray*}
(see e.g \cite{dv:calculDifferentiel},
\cite{dbm:central.bimodules}. On the other hand, there are
certain twisted derivations, called $\sigmatau$-derivations which are
usually not considered in the derivation-based approach but do play an
important role in certain deformations of Lie algebras as well as in
quantum groups. For example, in \cite{hls:sigmaderivation} and
\cite{ls:quasi-deformations}, a $q$-deformed Witt algebra was realized
through the use of $\sigmatau$-derivations and a family of
$q$-deformed Lie algebra of $\mathfrak{sl}_{2}(\mathbb{F})$ was
constructed using $\sigma$-derivations respectively.

A $\sigmatau$-derivation is a $\complex$-linear map $X:\A\to\A$ satisfying a
twisted Leibniz rule
\begin{eqnarray*}
	X(fg) = \sigma(f)X(g) + X(f)\tau(g),
\end{eqnarray*}
for algebra endomorphisms $\sigma, \tau : \A\to\A$.  In
  \cite{arnlind2022geometry}, a framework for noncommutative
  Riemmanian geometry is considered by constructing connections,
  curvature, and torsion over certain modules based on
  $\sigmatau$-derivations. In this framework, given a left
  $\A$-module $M$, a $\sigmatau$-connection $\nabla_{X}: M\to M,$
satisfies a twisted Leibniz rule
\begin{eqnarray}\label{twisted Leibniz of conn in intro}
	\nabla_{X}(fm) = \sigma(f)\nabla_{X}(m) + X(f)\tauh(m)
\end{eqnarray} 
for a $\sigmatau$-derivation $X: \A\to\A$ and a $\complex$-linear map
$ \tauh : M\to M$ (satisfying $\tauh(fm)=\tau(f)\tauh(m)$). For
instance, in the quest to explore connections on the quantum
$3$-sphere, \cite{ail:qdeformed} and \cite{ail:lc.spheres} constructed
connections satisfying \eqref{twisted Leibniz of conn in intro} with
respect to a set of $\sigmatau$-derivations $X_1,X_2,X_3$ appearing
naturally as deformations of the Lie algebra of vector fields on the
3-sphere.  These results motivated the construction of a general
framework for noncommutative geometry of twisted derivations developed
in \cite{arnlind2022geometry}, introducing the concept of a
$\sigmatau$-algebra together with a corresponding notion of
$\Sigma$-modules.

In the current work, we introduce the notion of symmetric
$\sigmatau$-derivations. A symmetric $\sigmatau$-derivation is a
$\sigmatau$-derivation that is also a $\tausigma$-derivation.  For
example, all $\sigmatau$-derivations on a commutative algebra are
symmetric $\sigmatau$-derivations. Based on such
$\sigmatau$-derivations, we introduce symmetric $\sigmatau$-algebras
and $\Sigma$-modules on which symmetric $\sigmatau$-connections are
constructed. To better understand inner and outer
$\sigmatau$-derivations, we introduce the concept of regular and
strongly regular pairs of endomorphisms.
This triggers interesting questions in related to Hochschild
cohomology. In degree one, the Hochschild cohomology group of an
associative algebra measures to what extent there exist outer
derivations of an associative algebra. For example, the first Hochschild
cohomology group of the matrix algebra $\mat$ in dimension one is zero,
implying that there are only inner derivations on a matrix
algebra. For $\sigmatau$-derivations on an associative algebra, we
introduce a $\sigmatau$-version of Hochschild cohomology for
describing the inner and outer $\sigmatau$-derivations on an
associative algebra.

In Section $2,$ we study symmetric $\sigmatau$-derivations on an
associative algebra and for example, prove that for a strongly regular
pair of endomorphisms, every symmetric $\sigmatau$-derivation is
inner. In Section 3, we define the $\sigmatau$-Hochschild cohomology
groups of an associative algebra in order to study inner and outer
$\sigmatau$-derivations. In Section $4,$ we study symmetric
$\sigmatau$-connections and show that for strongly regular
$\sigmatau$-algebras there exist a unique symmetric
$\sigmatau$-connection. In Section $5,$ we explore the curvature of a
symmetric $\sigmatau$-connection and its corresponding linearity
properties. Lastly in section $6,$ we illustrate the novel concepts
with examples based on commutative algebras.

\section{Symmetric $\sigmatau$-derivations}

\noindent
In this Section, we introduce a notion of symmetric
$\sigmatau$-derivations on associative algebras. These are
$\sigmatau$-derivations that are also $\tausigma$-derivations.  One
may think of symmetric $\sigmatau$-derivations as a generalization of
$\sigmatau$-derivations on commutative algebras (where every
  $\sigmatau$-derivation is symmetric) to noncommutative algebras. In
the following, we consider a unital associative algebra $\A$ over a
field $\K$ of characteristic zero. Let us recall some definitions.

\begin{definition}
  Let $\sigma,\tau\in\EndA$ and let $M$ be an $\A$-bimodule.  A
  $\K$-linear map $X:\A\to M$ is called a $\sigmatau$-derivation with
  values in $M$ if
  \begin{equation}
    X(fg) = \sigma(f)X(g) + X(f)\tau(g)
  \end{equation}
  for all $f, g \in \A.$
\end{definition}

\noindent
Let $\mbox{Der}_{\sigmatau}(\A)$ denote the set of
$\sigmatau$-derivations of $\A$. In general,
$\mbox{Der}_{\sigmatau}(\A)$ is not an $\A$-module, but it turns out
to be a module over $Z(\A)$, the center of $\A$.

\begin{proposition}
  $\mbox{Der}_{\sigmatau}(\A)$ is a $Z(\A)$-bimodule.
\end{proposition}

\begin{proof}
  Assume $X$ is a $\sigmatau$-derivation and let $f\in Z(\A).$ Set
  \begin{align*}
    (f\cdot X)(a) = fX(a)\quad\text{and}\quad
    (X\cdot f)(a) = X(a)f
  \end{align*}
  for $a\in\A$. It is clear that the above action defines a bimodule
  structure on $\mbox{Der}_{\sigmatau}(\A)$ as long as $f\cdot X$ and
  $X\cdot f$ are $\sigmatau$-derivations.  For $a, b \in \A,$ one has
  \begin{eqnarray*}
    &&f\cdot X(ab) = fX(ab) = f\sigma(a)X(b) + fX(a)\tau(b) = \sigma(a)fX(b) + fX(a)\tau(b)\\
    && \sigma(a)f\cdot X(b) + f\cdot X(a)\tau(b)
  \end{eqnarray*}
  showing that $f\cdot X$ is a $\sigmatau$-derivation. Similarly, one
  shows that $X\cdot f$ is a $\sigmatau$-derivation.
\end{proof}

\noindent
In the definition of $\sigmatau$-derivation, one obtains an
ordinary derivation (with values in $M$) when both $\sigma$ and $\tau$
are the identity maps on $\A$. A derivation $X$ is inner if it can be
written as $X(f)=[f,g_0]=m_0f-fm_0$ for some fixed $m_0\in M$.  These
usually give an idea of how noncommutative an algebra is. In the case
of $\sigmatau$-derivations, one defines inner derivations as follows.

\begin{definition}
  A $\sigmatau$-derivation $X$ with values in $M$ is called \emph{inner} if there exists
  $m_{0}\in M$ such that
  \begin{equation*}\label{inn sigmatau}
    X(f) = m_{0}\tau(f) - \sigma(f)m_{0}
  \end{equation*}
  for all $f \in \A.$ 
\end{definition}

\noindent
Note that although there are no inner derivations on a commutative
algebra, in general there exist inner (nontrivial)
$\sigmatau$-derivations on a commutative algebra if $\sigma\neq\tau$.

\begin{definition}
  A $\sigmatau$-derivation with values in $M$ is called
  \emph{symmetric} if it is also a $\tausigma$-derivation with values in $M.$
\end{definition}

\noindent
Let us look at a few examples of symmetric
$\sigmatau$-derivations. For arbitrary $\sigma,\tau\in\EndA$ the
  linear map $X:\A\to\A$ defined by
\begin{eqnarray}\label{ex sym inner sigmatau der}
  X(f) = \tau(f) - \sigma(f)
\end{eqnarray}
for $f \in \A$ is a symmetric $\sigmatau$-derivation. Indeed,
for $f, g \in \A,$ one has
\begin{eqnarray*}
  &&	X(fg) = \tau(f)\tau(g) - \sigma(f)\sigma(g) = \tau(f)\tau(g) - \sigma(f)\tau(g) + \sigma(f)\tau(g) -  \sigma(f)\sigma(g)\\
  &&= (\tau(f) - \sigma(f))\tau(g) + \sigma(f)(\tau(g) - \sigma(g)) = X(f)\tau(g) + \sigma(f)X(g)
\end{eqnarray*}
showing that $X$ is a $\sigmatau$-derivation. On the other hand, for
$f, g\in \A$ one has
\begin{eqnarray*}
  &&X(fg) = \tau(f)\tau(g) - \sigma(f)\sigma(g) = \tau(f)\tau(g) - \tau(f)\sigma(g) + \tau(f)\sigma(g) - \sigma(f)\sigma(g)\\
  && = \tau(f)(\tau(g) - \sigma(g)) + (\tau(f) - \sigma(f))\sigma(g) = \tau(f)X(g) + X(f)\sigma(g)
\end{eqnarray*}
showing that $X$ is also a $\tausigma$-derivation. Hence, $X$ is a
symmetric $\sigmatau$-derivation. The derivation given by \eqref{ex
  sym inner sigmatau der} is an inner $\sigmatau$-derivation which is
symmetric. Another example of a $\sigmatau$-derivation is the Jackson
derivative $D_{q}$ on the polynomial algebra $\complex[x].$ The
Jackson derivative is defined as
\begin{equation}\label{Jackson der}
  D_{q}(f(x)) = \dfrac{f(qx) - f(x)}{(q - 1)x}
\end{equation}
for $f(x) \in \complex[x]$ and $q \in \complex$ with $q\neq 1.$ For
$f(x), g(x) \in \complex[x],$ one has
\begin{eqnarray*}
  D_{q}(f(x)g(x)) = f(qx)D_{q}(g(x)) + D_{q}(f(x))g(x) = D_{q}(f(x))g(qx) + f(x)D_{q}(g(x)).
\end{eqnarray*}
With endomorphisms $\sigma(f)(x) = f(qx)$ and $\tau(f)(x) = f(x),$ the
Jackson derivative $D_{q}$ is a symmetric $\sigmatau$-derivation. This
an example of $\sigmatau$-derivation on the polynomial algebra
$\complex[x]$ which is not inner (see also Section~\ref{sec:ex.sigmatau.cohomology}).

On an associative algebra $\A,$ one can generalize \eqref{ex sym inner
  sigmatau der} to construct symmetric inner $\sigmatau$-derivations with values in a bimodule.
\begin{proposition}
  Let $\sigma,\tau\in\EndA$ and let $M$ be an $\A$-bimodule. If
  $m_{0}\in M$ such that $[m_{0}, \tau(f)] = [m_{0}, \sigma(f)] = 0$
  for all $f \in \A$ then $X:\A\to M$ given by
  \begin{eqnarray*}
    X(f) = m_{0}\tau(f) - \sigma(f)m_{0}
  \end{eqnarray*}
  for $f\in\A$, is a symmetric inner $\sigmatau$-derivation with values in $M$.
\end{proposition}

\begin{proof}
  By definition, $X$ is an inner $\sigmatau$-derivation.  We want to
  show that $X$ is also $\tausigma$-derivation.
  Since $[m_{0}, f] = 0 = [m_{0}, \sigma(f)]$
  for all $f \in \A,$ one has
  \begin{eqnarray*}
    X(f) = m_{0}\tau(f) - \sigma(f)m_{0} = \tau(f)m_{0} - m_{0}\sigma(f).
  \end{eqnarray*}
  For $f, g \in \A,$ one has
  \begin{eqnarray*}
    && X(fg) = \tau(f)\tau(g)m_{0} - \tau(f)m_{0}\sigma(g) + \tau(f)m_{0}\sigma(g) - m_{0}\sigma(f)\sigma(g)\\
    && = \tau(f)(\tau(g)m_{0} - m_{0}\sigma(g)) + (\tau(f)m_{0} - m_{0}\sigma(f))\sigma(g)\\
    && = \tau(f)X(g) + X(f)\sigma(g),
  \end{eqnarray*}
  showing that $X$ is a $\tausigma$-derivation. Since we have shown
  that $X$ is both $\sigmatau$-derivation and $\tausigma$-derivation,
  it is symmetric.
\end{proof}

Let us now introduce a notion of regularity for pairs of algebra
  endomorphisms. Given $\sigma,\tau\in\End(\A)$ we write
\begin{align}
  \delta(f) = \tau(f)-\sigma(f)
\end{align}
for $f\in\A$. In what follows, the properties of $\delta$ will be
important, and to this end we introduce the concept of regularity.
\begin{definition}
  Let $\sigma, \tau\in\End(\A)$. The pair $\sigmatau$ is
  called \emph{regular} if there exists $f\in \A$ such that
  $\delta(f) = \tau(f) - \sigma(f)$ is not a zero divisor. Moreover, the pair
  $\sigmatau$ is called \emph{strongly regular} if there exists
  $f\in\A$ such that $\delta(f)$ is invertible.
\end{definition}
\noindent
Let us now look at some examples of regular and strongly regular pairs
of endomorphisms. Let us take, for example, $\A = \complex[x],$ the
polynomial algebra in a variable $x.$ Since there are no zero divisors
in $\complex[x],$ any pair $\sigmatau$ such that $\sigma \neq \tau$ is
a regular pair of endomorphisms. Now, let $\sigma(x) = qx$ and
$\tau(x) = x$ for $q \in \complex, q \neq 1$. Since $\sigma \neq \tau$
the pair $\sigmatau$ is regular. Moreover, for a polynomial
$p(x) = \sum_{i\geq 0}a_{i}x^{i} \in \complex[x],$ \begin{eqnarray*}
  \delta(p)(x) = \sum_{i\geq 1}(1 - q^{i})a_{i}x^{i}
\end{eqnarray*}
is not invertible in $\complex[x].$ Thus the pair $\sigmatau$ such
that $\sigma(x) = qx$ and $\tau(x) = x,$ is not a strongly regular
pair of endomorphisms of $\complex[x].$ On the other hand, given
$\sigma(x) = x - h$ and $\tau(x) = x,$ for $h\in\complex$ ($h\neq 0$), then
\begin{eqnarray*}
	\delta(x) = h
\end{eqnarray*}
is invertible in $\complex[x],$ implying that the pair $\sigmatau$
such that $\sigma(x) = x - h$ and $\tau(x) = x$ is a strongly regular
pair of endomorphisms over $\complex[x].$

In the following Lemma, we give a property of symmetric
$\sigmatau$-derivations that motivates a relationship between regular
pairs of endomorphisms and symmetric $\sigmatau$-derivations.
\begin{lemma}\label{lemma:Xd.dX}
	If $X$ is a symmetric $\sigmatau$-derivation with values in the $\A$-bimodule $M$, then
	\begin{equation}
		X(f)\delta(g) = \delta(f)X(g)
	\end{equation}
	for all $f,g\in\A$, where $\delta = \tau-\sigma$.
\end{lemma}

\begin{proof}
	Assume that $X$ is a symmetric $\sigmatau$-derivation. For $f, g \in \A,$ one has
	\begin{eqnarray*}
		&&0 = X(fg) - \sigma(f)X(g) - X(f)\tau(g)\\
		&&= \tau(f)X(g) + X(f)\sigma(g) - \sigma(f)X(g) - X(f)\tau(g)\\
		&& = (\tau(f) - \sigma(f))X(g) + X(f)(\sigma(g) - \tau(g))\\
		&&= \delta(f)X(g) - X(f)\delta(g),
	\end{eqnarray*}
	implying that 
	\begin{equation*}
		X(f)\delta(g) = \delta(f)X(g).\qedhere
	\end{equation*}
\end{proof}
\noindent
\begin{proposition} \label{str reg symm der} Let $\sigmatau$ be a
  strongly regular pair of endomorphisms of $\A$ and let $X$ be a
  symmetric $\sigmatau$-derivation with values in the $\A$-bimodule
  $M$. Then there exists a unique $m_{0} \in M$ such that
  \begin{eqnarray*}
    X(f) = m_{0}\tau(f) - \sigma(f)m_{0}
  \end{eqnarray*}
  and 
  \begin{eqnarray*}
    [m_{0}, \tau(f)] = [m_{0}, \sigma(f)] = 0
  \end{eqnarray*}
  for all $f \in \A$.
\end{proposition}

\begin{proof}
  Since $X$ is a symmetric $\sigmatau$-derivation, one concludes from
  Lemma~\ref{lemma:Xd.dX} that
  \begin{align}\label{eq:XddX}
    X(f)\delta(g) = \delta(f)X(g)
  \end{align}
  for all $f,g\in\A$. Since $\sigmatau$ is assumed to be a strongly
  regular pair of endomorphisms, there exists $g_0\in\A$ such that
  $\delta(g_0)$ is invertible. It then follows from \eqref{eq:XddX}
  that
  \begin{equation}\label{eq:X.in.terms.of.delta}
    X(f) = \delta(f)X(g_0)\delta(g_0)^{-1}
  \end{equation}
  for all $f\in\A$. Set $m_0=X(g_0)\delta(g_0)^{-1}$, use
  \eqref{eq:X.in.terms.of.delta} to compute
  \begin{align*}
    0 &= X(fg)-\sigma(f)X(g)-X(f)\tau(g)
        = \delta(fg)m_0-\sigma(f)\delta(g)m_0-\delta(f)m_0\tau(g)\\
      &=\paraa{\tau(f)\tau(g)-\sigma(f)\sigma(g)}m_0-\sigma(f)\paraa{\tau(g)-\sigma(g)}m_0
        -\paraa{\tau(f)-\sigma(f)}m_0\tau(g) \\
      &= \tau(f)\tau(g)m_0-\sigma(f)\tau(g)m_0-\tau(f)m_0\tau(g)+\sigma(f)m_0\tau(g)\\
      &= \tau(f)[\tau(g),m_0]-\sigma(f)[\tau(g),m_0] = \delta(f)[\tau(g),m_0],
  \end{align*}
  and setting $f=g_0$, implying that $\delta(g_0)$ is invertible,
  gives $[\tau(g),m_0]=0$ for all $g\in\A$. Similarly, the identity
  \begin{align*}
    0 &= X(fg)-\tau(f)X(g)-X(f)\sigma(g)
  \end{align*}
  implies that $[\sigma(g),m_0]=0$ for all $g\in\A$. It follows that
  \begin{align*}
    X(f) = \delta(f)m_0 = m_0\delta(f) = m_0\tau(f) - \sigma(f)m_0.
  \end{align*}
  Finally, let us prove that $m_0$ is unique. To this end, assume that
  \begin{align*}
    X(f)=\delta(f)m_0=\delta(f)n_0 
  \end{align*}
  for all $f\in\A$, implying that $\delta(f)(m_0-n_0)=0$ for all
  $f\in\A$. Setting $f=g_0$ gives $m_0=n_0$ since $\delta(g_0)$ is
  invertible.
\end{proof}

\begin{corollary}\label{cor:strongly.regular.inner}
  If $\sigmatau$ is a strongly regular pair of endomorphisms of $\A$ then
  every symmetric $\sigmatau$-derivation with values in an $\A$-bimodule $M$ is inner.
\end{corollary}

\section{$\sigmatau$-Hochschild Cohomology}

\noindent
The notion of Hochschild cohomology was first introduced in
\cite{hochschild}. Hochschild cohomology theory in low dimension is
concerned with derivations of associative algebras. For instance, the
first cohomology group determines the outer derivations of the
algebra.
In this Section, we introduce a $\sigmatau$-version of Hochschild
cohomology, which we refer to as $\sigmatau$-Hochschild
cohomology. The $\sigmatau$-Hochschild cohomology will be defined as
the Hochschild cohomology with coefficients in a particular twisted
bimodule, which we denote as $\Msigmatau.$ In this case of the
Hochschild cohomology, $1$-cocycles are $\sigmatau$-derivations and
$1$-coboundaries are inner $\sigmatau$-derivations. Therefore, first
$\sigmatau$-Hochschild cohomology group, which we denote by
$H^{1}_{\sigmatau}(\A, M)$, determines the outer
$\sigmatau$-derivations. Let us first look at the preliminaries before
we go into details.

\subsection{Hochschild cohomology preliminaries}
Let
$\A^{\otimes n} =
\A\otimes_{\complex}\cdots\otimes_{\complex} \A$
be the $n$-fold tensor product of $\A$ and let $M$ be an
$\A$-bimodule. A $\complex$-linear map $\omega : \A^{\otimes n} \to M$
is called an $n$-cochain of an algebra $\A$ with values in $M$ and the
linear space of all $n$-cochains is denoted $C^{n}(\A, M)$ for
$n \geq 1.$ For $n = 0,$ let $C^{0}(\A, M) = M.$
The Hochschild cochain complex of $\A$ with coefficients in $M$,
denoted $C^{\ast}(\A, M, \delta)$ is the sequence of linear
spaces of cochains
\begin{eqnarray*}
  0 \to M \to C^{1}(\A, M) \to C^{2}(\A, M)\to \cdots \to C^{n}(\A, M)\to C^{n+1}(\A, M)\to \cdots
\end{eqnarray*}
together with the differentials $\delta_{n}:C^{n}(\A, M) \to C^{n + 1}(\A, M)$, defined by
\begin{eqnarray*}
  (\delta_{0}m)(a) = ma -am
\end{eqnarray*}
for $m \in M$ and $a \in A$, and 
\begin{align*}
  (\delta_{n}\omega)(a_{1}, \ldots, a_{n+1}) = a_{1}&\omega(a_{2}, \ldots, a_{n+1})
  + \sum_{j=1}^{n}(-1)^{j}\omega(a_{1}, \ldots, a_{j}a_{j+1}, \ldots, a_{n+1})\\
                                             &+ (-1)^{n+1}\omega(a_{1}, \ldots, a_{n})a_{n+1}
\end{align*}
for $\omega \in C^{n}(\A, M)$ and $n \geq 1$. One shows that $\delta_{n} \circ \delta_{n-1} = 0$.

An $n$-cochain $\omega$ is called $n$-cocycle if $\delta \omega = 0,$
and it is called an $n$-coboundary if there exist an $(n-1)$-cochain
$\rho$ such that $\omega = \delta \rho.$ The set of $n$-cocycles form
a subgroup of the additive group of $n$-cochains. The $n$-dimensional
cohomology group of $\A$ with coefficients in $M$, denoted by
$H^{n}(\A, M),$ is the group of $n$-dimensional cocycles
modulo the subgroup of $n$-dimensional coboundaries. Namely, denoting
by $Z^{n}(\A, M)$ the set of $n$-cocycles, i.e.
\begin{eqnarray*}
  Z^{n}(\A, M) = \{\omega \in C^{n}(\A, M):\delta\omega = 0\}
\end{eqnarray*}
and denoting by $B^{n}(\A, M)$ the set of $n$-coboundaries, i.e.
\begin{eqnarray*}
  B^{n}(\A, M) = \{\omega \in C^{n}(\A, M) : \omega = \delta\rho ,\hspace{0.1 cm} \mbox{ for} \hspace{0.1 cm} \rho \in C^{n-1}(\A, M)\},
\end{eqnarray*}
one defines
\begin{eqnarray*}
	H^{n}(\A, M) = Z^{n}(\A, M)/B^{n}(\A, M).
\end{eqnarray*}
Now, let us examine the cohomology groups for the cases when
$n = 0, 1.$ For $n = 0,$ the Hochschild cohomology group is given by
\begin{eqnarray*}
  H^{0}(\A, M) = \mbox{ker}(\delta_{0}) = \{m \in M: ma = am, \hspace{0.1 cm}\mbox{for}\hspace{0.1 cm} a \in \A\},
\end{eqnarray*}
which is simply the center of $M.$ For $n =1,$ one considers
\begin{eqnarray*}
  (\delta_{1}\omega)(a_{1}, a_{2}) = a_{1}\omega(a_{2}) - \omega(a_{1}a_{2}) + \omega(a_{1})a_{2}
\end{eqnarray*}
for $\omega\in C^1(\A,M)$ and the boundary map
\begin{eqnarray*}
	(\delta_{0}m)(a_{1}) = [m, a_{1}],
\end{eqnarray*}
for $a_{1}, a_{2} \in \A, m \in M.$  Given that $\omega \in \mbox{ker}(\delta_{1}),$ then one has
\begin{eqnarray*}
	\omega(a_{1}a_{2}) = a_{1}\omega(a_{2}) + \omega(a_{1})a_{2}
\end{eqnarray*}
implying that the linear map $\omega: \A\to M$ is a derivation of $\A$
with values in $M.$ One writes
\begin{eqnarray*}
  Z^{1}(\A, M) = \{\omega\in C^1(\A,M) : \omega(ab) = a\omega(b) + \omega(a)b 
  \quad \forall\, a, b \in\A\}
  = \mbox{Der}(\A, M),
\end{eqnarray*}
where $\mbox{Der}(\A, M)$ denotes the set of all derivations of $\A$
with values in $M.$ For $\omega$ in the image of $\delta_{0},$ one has
\begin{eqnarray*}
  B^{1}(\A, M) = \{\omega\in C^1(A,M): \omega = (\delta_{0}m)(a)
  = [m, a], \forall\, a \in \A\} = \mbox{Inn}(\A, M),
\end{eqnarray*}
where $\mbox{Inn}(\A, M)$ denotes the set of inner derivations of
$\A$ with values in $M.$ Thus, it follows that
\begin{eqnarray*}
  H^{1}(\A, M) = Z^{1}(\A, M)/B^{1}(\A, M) = \mbox{Der}(\A, M)/\mbox{Inn}(\A, M),
\end{eqnarray*}
represents the set of outer derivations of $\A$ with values in $M.$

\subsection{$\sigmatau$-Hochschild cohomology}
Let $\A$ be an associative algebra and $M$ be an $\A$-bimodule. Given
$\sigma,\tau\in\End(\A)$, we will construct a new $\A$-bimodule
structure on $M$, denoted by $M_{(\sigma, \tau)}$, by twisting the
action on $M$ by the endomorphisms $\sigma$ and $\tau.$ Namely, let
\begin{eqnarray*}
	f\cdot m = \sigma(f)m, \quad m\cdot f = m\tau(f)
\end{eqnarray*}
for $f \in \A$ and $m \in M.$ It is easy to show that $\Msigmatau$ is again
an $\A$-bimodule. We want to compute the Hochschild cohomology group
of $\A$ with values in $\Msigmatau.$ The boundary map
$\delta_{n}: C^{n}(\A, \Msigmatau) \to C^{n+1}(\A, \Msigmatau)$ is now
obtained as
\begin{eqnarray*}
	(\delta_{0}m)(a) = m\cdot a - a\cdot m = m\tau(a) - \sigma(a)m
\end{eqnarray*}
for $m \in \Msigmatau$ and $a \in \A,$ and 
\begin{align*}
  (\delta_{n}\omega)(a_{1}, \ldots,a_{n+1}) &=
     a_{1}\cdot \omega(a_{2}, \ldots, a_{n+1})
     + \sum_{j=1}^{n} (-1)^{j}\omega(a_{1}, \ldots, a_{j}a_{j+1},\ldots,a_{n+1})\\
  &\qquad+ (-1)^{n+1}\omega(a_{1}, \ldots,a_{n})\cdot a_{n+1}\\
  &=\sigma(a_{1})\omega(a_{2}, \ldots, a_{n+1}) 
     + \sum_{j=1}^{n}(-1)^{j}\omega(a_{2}, \ldots, a_{j}a_{j+1}, \ldots, a_{n+1})\\
  &\qquad+(-1)^{n+1}\omega(a_{1}, \ldots, a_{n})\tau(a_{n+1})
\end{align*}
for $n \geq 1$. In particular, for $n=1$ one obtains
\begin{align*}
  (\delta_1\omega)(a_1,a_2) = \sigma(a_1)\omega(a_2) +\omega(a_1)\tau(a_2)-\omega(a_1a_2)
\end{align*}
implying that a $1$-cocycle $\omega \in C^{1}(\A, \Msigmatau)$ is a 
$\sigmatau$-derivation with values in $M$. Let us denote
by $\mbox{Der}_{\sigmatau}(\A, M)$ and $\mbox{Inn}_{\sigmatau}(\A, M)$
the set of all $\sigmatau$-derivations and the set of all inner
$\sigmatau$-derivations of $\A$ with values in $M$ respectively. Then
one sets
\begin{eqnarray*}
  &Z^{n}_{\sigmatau}(\A, M) = Z^{n}(\A, \Msigmatau) \\
  &B^{n}_{\sigmatau}(\A, M) = B^{n}(\A, \Msigmatau) 
\end{eqnarray*}
noting that
\begin{eqnarray*}
  Z^{1}_{\sigmatau}(\A, M)& = \mbox{Der}_{\sigmatau}(\A, M),\\
  B^{1}_{\sigmatau}(\A,  M)& = \mbox{Inn}_{\sigmatau}(\A, M).
\end{eqnarray*}
Hence, the Hochschild
cohomology groups of $\A$ with values in $\Msigmatau$ are defined as
\begin{eqnarray*}
	H^{n}_{\sigmatau}(\A, M) =  Z^{n}_{\sigmatau}(\A, M)/B^{n}_{\sigmatau}(\A,  M) . 
\end{eqnarray*}

\noindent
Let us now consider the $\sigmatau$-Hochschild cohomology
groups, for $n = 0, 1.$ For $n = 0$ one has
\begin{eqnarray*}
  H^{0}_{\sigmatau}(\A, M) = \{m\in \Msigmatau: m\tau(a) = \sigma(a)m, \hspace{0.1 cm}\mbox{for} \hspace{0.1 cm} a \in \A\}.
\end{eqnarray*}
\begin{proposition}\label{0Hochschild(A)}
  Let $\A$ be a commutative algebra and $\sigmatau$ be a regular
  pair. Then
  \begin{eqnarray*}
    H^{0}_{\sigmatau}(\A, \A) =  0.
  \end{eqnarray*}
\end{proposition}
\begin{proof}
  Let $b \in H_{\sigmatau}^{0}(\A, \A),$ then $b\delta(a) = 0$ for
  $a \in \A$ (using that $\A$ is commutative). Moreover, since $\sigmatau$ is
  a regular pair, there exists $a_{0} \in \A$ such that $\delta(a_{0})$ is not
  a zero divisor. Then
  \begin{eqnarray*}
    b\delta(a_{0}) = 0 \implies b = 0.
  \end{eqnarray*}
  Hence, $H_{\sigmatau}^{0}(\A, \A) = 0.$
\end{proof}

\noindent
For $n = 1,$ $(\delta_{1}\omega)(a_{1}, a_{2}) = 0$ implies that 
\begin{eqnarray*}
  \omega(a_{1}a_{2}) = \sigma(a_{1})\omega(a_{2}) + \omega(a_{1})\tau(a_{2})
\end{eqnarray*}
showing that $\omega \in Z^1_{\sigmatau}(\A, M)$ is a
$\sigmatau$-derivation. For $\omega \in C^{1}(\A, \Msigmatau)$ such
that
\begin{eqnarray*}
	\omega = (\delta_{0}m)(a) = m\tau(a) - \sigma(a)m
\end{eqnarray*}
implies that $\omega \in B_{\sigmatau}^{1}(\A, M)$ is an inner
$\sigmatau$-derivation. Hence
\begin{eqnarray*}
	H_{\sigmatau}^{1}(\A, M) =  \mbox{Der}_{\sigmatau}(\A, M)/\mbox{Inn}_{\sigmatau}(\A, M)
\end{eqnarray*}
constitute all outer $\sigmatau$-derivations of $\A$ with values in
$M.$ When $M=\A$ we shall simply write
$H^1_{\sigmatau}(\A) = H^{1}_{\sigmatau}(\A, \A)$. Note that if
$H^{1}_{\sigmatau}(\A) = 0$ then every $\sigmatau$-derivation is
inner.

Let us now consider a few cases where one can determine
  the $\sigmatau$-cohomology in lower degrees. First of all, it
  follows immediately from Corollary~\ref{cor:strongly.regular.inner}
  that $H^1_{\sigmatau}(\A,M)$ vanishes for strongly regular symmetric
  $\sigmatau$-derivations, which we state as follows.

\begin{proposition}
  Let $\sigma,\tau\in\End(\A)$ be such that $\sigmatau$ is a strongly
  regular pair of endomorphisms. If every $\sigmatau$-derivation on
  $\A$ with values in $M$ is symmetric, then
  $H^1_{\sigmatau}(\A,M)=H^1_{\tausigma}(\A,M)=0$.
\end{proposition}

\noindent
In particular, for a commutative algebra $\A$, every
$\sigmatau$-derivation is symmetric and one can conclude that
$H^1_{\sigmatau}(\A)=0$ if $\sigmatau$ is a strongly regular pair of
endomorphisms.

\subsection{Examples of $\sigmatau$-Hochchild cohomology}\label{sec:ex.sigmatau.cohomology}
Let $\A = \mathbb{K}[x],$ be the polynomial algebra in the
indeterminate $x$ over the field $\mathbb{K}.$ Both the zero and one
dimensional Hochschild cohomology of $\mathbb{K}[x]$ is isomorphic to
$\mathbb{K}[x],$ i.e., $H^{0}(\mathbb{K}[x]) \simeq \mathbb{K}[x]$ and
$H^{1}(\mathbb{K}[x]) \simeq \mathbb{K}[x]\frac{d}{dx}$. (see for example
\cite{witherspoon}). In this section we describe the $\sigmatau$-Hochschild cohomology groups in degree 0 and 1.

In degree 0, one can apply Proposition~\ref{0Hochschild(A)}, assuming
that $\sigma\neq\tau$. In that case, there exists a polynomial
$p_{0}\in\K[x]$ such that $\delta(p_{0}) \neq 0$ implying that
$\delta(p_0)$ is not a zero divisor since $\K[x]$ is an integral
domain. Hence
\begin{eqnarray*}
	H^{0}_{\sigmatau}(\mathbb{K}[x]) = 0.
\end{eqnarray*}

Next, let us consider $H^1_{\sigmatau}(\K[x])$.

\begin{proposition}\label{outer der poly alg}
  Let $\sigma,\tau\in\End(\K[x])$ and let $\delta=\tau-\sigma$. Then
  \begin{align*}
    H^{1}_{\sigmatau}(\mathbb{K}[x]) = \mathbb{K}[x]/I_{\delta},
  \end{align*}
  where 
  \begin{eqnarray*}
    I_{\delta} = \{p(x)\delta(x): p(x) \in \mathbb{K}[x]\}	
  \end{eqnarray*}
  is the ideal generated by $\delta(x)$.
\end{proposition}

\begin{proof}
  In $\mathbb{K}[x],$ a $\sigmatau$-derivation $X$ is completely determined by 
  \begin{eqnarray*}
    X(x) = f_{0}(x) \in \mathbb{K}[x]
  \end{eqnarray*}
  and, conversely, any such choice of $f_{0} \in \mathbb{K}[x]$ gives
  a $\sigmatau$-derivation by using the twisted Leibniz rule, implying
  that
  $Z^{1}_{\sigmatau}(\mathbb{K}[x]) \simeq
  \mathbb{K}[x]$, where the isomorphism is provided by
  $X\mapsto X(x)=f_0(x)$. Now if
  $X \in B^{1}_{\sigmatau}(\mathbb{K}[x])$ then
  \begin{eqnarray*}
    X(x) = g_{0}(x)(\tau(x) - \sigma(x)) = g_{0}(x)\delta(x), 
  \end{eqnarray*}
  for some $g_{0} \in \mathbb{K}[x]$, implying that 
  \begin{eqnarray*}
    B^{1}_{\sigmatau}(\mathbb{K}[x]) \simeq I_{\delta}.
  \end{eqnarray*}
  Hence
  \begin{equation*}
    H^{1}_{\sigmatau}(\mathbb{K}[x])
    = Z^{1}_{\sigmatau}(\mathbb{K}[x])/B^{1}_{\sigmatau}(\mathbb{K}[x])
    =  \mathbb{K}[x]/I_{\delta}.\qedhere
  \end{equation*}
\end{proof}

\noindent
To illustrate the above result, let us apply it to a particular case.

\begin{proposition}\label{H1(k[x])}
  Let $\sigma, \tau$ be endomorphisms of the polynomial algebra
  $\mathbb{K}[x]$ given by $\sigma(x) = qx$ and $\tau(x) = x$ for
  $q \neq 0, 1.$ Then
  \begin{eqnarray*}
    H^{1}_{\sigmatau}(\mathbb{K}[x], \mathbb{K}[x]) = \{\lambda 1 : \lambda \in \mathbb{K}\}.
  \end{eqnarray*}
\end{proposition}

\begin{proof}
  Using the definition of $\sigma$ and $\tau,$ one obtains
  \begin{eqnarray*}
    \delta(x) = \tau(x) - \sigma(x) = (1 - q)x
  \end{eqnarray*}
  and note that 
  \begin{eqnarray*}
    \left(\dfrac{1}{1-q}\right)\delta(x) = x, \left(\dfrac{x}{1-q}\right)\delta(x) = x^{2}, \ldots, \left(\dfrac{x^{n-1}}{1-q}\right)\delta(x) = x^{n},
  \end{eqnarray*}
  are all in the ideal
  $I_{\delta} = \langle p(x)\delta(x)| p(x) \in \mathbb{K}[x]\rangle$
  except $1.$ Therefore, $I_{\delta}$ is the polynomial algebra
  $\mathbb{K}[x]$ without $\lambda1$ for $\lambda \in \mathbb{K}.$
  Using Proposition \ref{outer der poly alg}, the outer
  $\sigmatau$-derivations of $\mathbb{K}[x]$ for the endomorphisms
  $\sigma(x) = qx$ and $\tau(x) = x$ are
  \begin{equation*}
    \mathbb{K}[x]/I_{\delta} = \{\lambda1| \lambda \in \mathbb{K}\}.\qedhere
  \end{equation*}
\end{proof}

\begin{corollary}\label{Jack out der}
	The Jackson derivative is an outer $\sigmatau$-derivation of $\mathbb{K}[x].$
\end{corollary}
\begin{proof}
  In $\mathbb{K}[x]$, every $\sigmatau$-derivation $X$ is determined
  by $X(x) = f_{0}(x) \in \mathbb{K}[x].$ Let $X$ be the Jackson
  derivative. It is easy to show that $X$ is a $\sigmatau$-derivation
  with $\sigma(x) = qx$ and $\tau(x) = x$ and $X(x) = 1.$ Since the
  Jackson derivative is determined by $f_{0}(x) = 1$ then by
  Proposition \ref{H1(k[x])}, it is an outer derivation of
  $\mathbb{K}[x].$
\end{proof}

\noindent
Next, let $\A = \mat$ be the algebra of $(N \times N)$ complex
matrices. If $\sigma,\tau\in\End(\A)$ are automorphisms, one can
easily describe $H^0_{\sigma\tau}(\mat)$.

\begin{proposition}\label{prop:H0.matrix}
  If $\sigma,\tau\in\Aut(\mat)$ then $H^0_{\sigmatau}\paraa{\mat}\simeq\complex$.
\end{proposition}

\begin{proof}
  Since $\mat$ is a central simple algebra, the Skolem-Noether theorem
  implies that all automorphisms are inner; that is, there exist
  $U,V\in\mat$ (unique up to multiplication by a complex number) such
  that $\tau(A)=U^{-1}AU$ and $\sigma(A)=V^{-1}AV$ for all
  $A\in\mat$. By definition
  \begin{align*}
    H^0_{\sigmatau}(\mat) = \{A\in\mat:A\tau(B)=\sigma(B)A\quad\forall\, B\in\mat\},
  \end{align*}
  and one computes
  \begin{align*}
    A\tau(B)=\sigma(B)A\equivalent
    AU^{-1}BU = V^{-1}BVA\equivalent
    [VAU^{-1},B]=0.
  \end{align*}
  For the above relation to be true for all $B\in\mat$ one necessarily
  has that $VAU^{-1}=z\mid_N$ for some $z\in\complex$, where $\mid_N$
  denotes the identity matrix. Thus, it follows that
  \begin{equation*}
    H^0_{\sigmatau}\paraa{\mat} = \{zV^{-1}U: z\in\complex\}\simeq\complex.\qedhere
  \end{equation*}
\end{proof}

\noindent
It is well known that all derivations on $\mat$ are inner, implying
that
\begin{equation*}
  H^1_{\sigmatau}(\mat)=0.
\end{equation*}
Moreover, it is a fact that $H^n(\mat,M)=0$ for $n\geq 1$ and
arbitrary bimodule $M$ (see e.g. \cite{MR1945297}), implying the
following result.

\begin{proposition}
  Let $\mat$ denote the algebra of $N\times N$ complex matrices, and let $M$ be a $\mat$-bimodule. Then
  \begin{eqnarray*}
    H^{n}_{\sigmatau}(\mat, M) = 0
  \end{eqnarray*}
  for $n\geq 1$.
\end{proposition}

\begin{proof}
  Since for any bimodule $M$ and $n\geq 1$ one has $H^{n}(\mat, M) = 0$, it follows that
  \begin{equation*}
    H^{1}_{\sigmatau}(\mat, M) = H^{1}(\mat, M_{\sigmatau}) = 0.\qedhere
  \end{equation*}
\end{proof}

\noindent
In particular, the above result implies that every $\sigmatau$-derivation on $\mat$ is inner.

\section{Symmetric $(\sigma,\tau)$-algebras}

\noindent
The framework to study $\sigmatau$-connections on $\sigmatau$-algebras
has been studied in \cite{arnlind2022geometry}, \cite{ail:qdeformed}
and \cite{ail:lc.spheres}. In this Section, we extend the framework to
accommodate symmetric $\sigmatau$-derivations already introduced in
Section $2.$ To this effect, we introduce symmetric
$\sigmatau$-algebras and symmetric $\sigmatau$-connections and study
how the strongly regular condition of endomorphisms in a
$\sigmatau$-algebra influence such connections.  Let us recall a few basic 
definitions from \cite{arnlind2022geometry}.

\begin{definition}
  A $(\sigma, \tau)$-algebra $\Sigma = (\A, \{X_{a}\}_{a \in I})$ is a
  pair where $\A$ is an associative algebra and $X_{a}$ is a
  $(\sigma_{a}, \tau_{a})$-derivation of $\A$ for $a \in I.$
\end{definition}

\begin{definition}
  For a $\sigmatau$-algebra $\Sigma=(\A,\{X_a\}_{a\in I}),$ we let
 $
    \TSigma \subseteq \mbox{Hom}_{\complex}(\A, \A)
  $
  be the vector space generated by $\{X_a\}_{a\in I}$. We call
  $\TSigma$ the \emph{tangent space of $\Sigma$}.
\end{definition}

\noindent
Recall that a pair $\sigmatau$ of endomorphisms is strongly regular if
there exists $f_{0} \in \A$ such that
$\delta(f_{0})=\tau(f_0)-\sigma(f_0)$ is invertible.  In a
$\sigmatau$-algebra, there are several pairs of endomorphisms, and for
each $\sigmatau$-derivation $X_{a}$ in the $\sigmatau$-algebra, there
are pairs $(\sigma_{a}, \tau_{a})$ for $a \in I.$ We would like to
consider the case that all the pairs of endomorphisms in the
$\sigmatau$-algebra are simultaneously strongly regular, i.e., there
exist $f_{a} \in \A$ for $a \in I$ such that
$\delta_{a}(f_{a}) = \tau_{a}(f_{a}) - \sigma_{a}(f_{a})$ is
invertible. We shall call a $\sigmatau$-algebra with this property
strongly regular. Similarly, if each $\sigmatau$-derivation in the
$\sigmatau$-algebra is symmetric, we shall call the
$\sigmatau$-algebra symmetric. In the following, we give precise
definitions.

\begin{definition}
  A $(\sigma, \tau)$-algebra $\Sigma = (\A, \{X_{a}\}_{a \in I})$ is
  called \emph{(strongly) regular} if $(\sigma_a,\tau_a)$ is a
  (strongly) regular pair of endomorphisms of $\A$ for all $a\in I$.
\end{definition}

\begin{definition}
  A symmetric $(\sigma, \tau)$-algebra
  $\Sigma = (\A, \{X_{a}\}_{a \in I})$ is a $\sigmatau$-algebra such
  that $X_{a}$ is a symmetric $(\sigma_{a}, \tau_{a})$-derivation for
  $a \in I$.
\end{definition}

\noindent
Note that if $\A$ is a commutative algebra, then any
$\sigmatau$-algebra $\Sigma = (\A, \{X_{a}\}_{a\in I})$ is symmetric
since $\sigmatau$-derivations on commutative algebras are
symmetric. For such an algebra to be a regular $\sigmatau$-algebra,
one may consider the case when the algebra has no zero divisors and
$\sigma_{a} \neq \tau_{a}$ for all $a\in I.$ Again, this is possible
because then
$\delta_{a}(f_a) = \tau_{a}(f_a) - \sigma_{a}(f_A) \neq 0$ for some
$f_a\in\A$.
To obtain a strongly regular algebra, one takes for example, the algebra $\K[x]$ of
polynomials in one variable $x$ and assumes that $\sigma_{a} = \sigma$
with $\sigma(p)(x) = p(x - h)$ for $h\in \K$ and $p(x)\in \K$ and
assume that $\tau_{a} = \tau$ with $\tau(p)(x) = p(x).$ In this case
each pair $(\sigma_{a}, \tau_{a})$ is strongly regular in $\K[x].$

Let us now recall the definition of modules over $\sigmatau$-algebras.

\begin{definition}
  Let $\Sigma=(\A,\{X_a\}_{a\in I})$ be a $\sigmatau$-algebra. A
  \emph{left $\Sigma$-module $(M,\{(\sigmah_a,\tauh_a)\}_{a\in I})$}
  is a left $\A$-module $M$ together with $\complex$-linear maps
  $\sigmah_a,\tauh_a:M\to M$ such that
  \begin{align*}
    &\sigmah_a(fm) = \sigma_a(f)\sigmah(m)\\
    &\tauh_a(fm) = \tau_a(f)\tauh_a(m)
  \end{align*}
  for $f\in\A$, $m\in M$ and $a\in I$.
\end{definition}

\noindent Similarly, one defines right $\Sigma$-modules as well as
$\Sigma$-bimodules (cf. \cite{arnlind2022geometry}).
Let now us recall a few examples of $\Sigma$-modules.

\begin{example}\label{free Sig mod}
  Let $\Sigma = (\A, \{X_{a}\}_{a\in I})$ be a $\sigmatau$-algebra and
  let $\A^{n}$ be a free left $\A$-module with basis $\{e_{i}\},$
  $i= 1, \ldots, n.$ One introduces a canonical $\Sigma$-module
  structure on $\A^{n}$ by setting
  \begin{eqnarray*}
    \sigmah_{a}^{0}(m) = \sigma_{a}(m^{i})e_{i}, \quad \tauh_{a}^{0}(m) = \tau_{a}(m^{i})e_{i}
  \end{eqnarray*}
  where $m = m^{i}e_{i} \in \A^{n}$ (summation over $i$ from $1$ to
  $n$ is implied). Now, let $f\in \A$ and $m\in \A^{n},$ one has
  \begin{eqnarray*}
    &&\sigmah_{a}^{0}(fm) = \sigma_{a}(fm^{i})e_{i} = \sigma_{a}(f)\sigma_{a}(m^{i})e_{i} = \sigma_{a}(f)\sigmah_{a}^{0}(m)\\
    &&\tauh_{a}^{0}(fm) = \tau_{a}(fm^{i})e_{i} = \tau_{a}(f)\tau_{a}(m^{i})e_{i} = \tau_{a}(f)\tauh_{a}^{0}(m),
  \end{eqnarray*}
  showing that $(\A^{n}, \{(\sigmah_{a}^{0}, \tauh_{a}^{0})\}_{a\in i})$ is a left $\Sigma$-module.
\end{example}

\noindent
The example above is an example of a free $\Sigma$-module. Let us now
look an example of a $\Sigma$-module which is not free.
\begin{example}

  Let $\A = \mat,$ the algebra of complex $(N\times N)$-matrices and
  let $\Sigma = (\A, \{X_{a}\}_{a=1}^{n})$ where $X_{a}:\mat\to\mat$
  is given by
  \begin{eqnarray*}
    X_{a}(A) = A - \sigma_{a}(A)
  \end{eqnarray*}
  with endomorphism $\sigma_{a}(A) = U_{a}AU_{a}^{-1}$ for
  $A \in \mat$ and some complex $(N\times N)$-matrices $U_{a}$ such
  that $[U_{a}, U_{b}] = 0$ for $a, b = 1, \ldots, n.$ Then $X_a$
    is a $(\sigma_a,\Id)$-derivation and
    $(\mat, \{(\sigma_{a}, \Id)\}_{a = 1}^{n})$ is a free
    $\Sigma$-module.

    According to \cite{arnlind2022geometry}, applying an appropriate
    projector to any $\Sigma$-module, one obtains a projective
    $\Sigma$-module. Let $p = v_{0}v_{0}^{\dagger}$ for
    $v_{0} \in \complex^{N}$ with $|v_{0}| = 1$ be the projector.
    Then $(\matp, \{(\sigmah_{a}, \tauh_{a})\}_{a = 1}^{n})$, with
    $\sigmah_{a}(A) = \sigma_{a}(A)p$ and $\tauh_{a}(A) = Ap$, is a
    projective $\Sigma$-module and one notes that $\matp$ is
    isomorphic to $\complex^N$, which is not a free left
    $\mat$-module.

\end{example}

\noindent
Now, let us recall the definition of $\sigmatau$-connections on $\Sigma$-modules.
\begin{definition}
	Let $\Sigma=(\A,\{X_a\}_{a\in I})$ be a $\sigmatau$-algebra and let
	$(M,\{(\sigmah_a,\tauh_a)\}_{a\in I})$ be a left $\Sigma$-module. A
	left \emph{$\sigmatau$-connection on $M$} is a map
	$\nabla:\TSigma\times M\to M$ satisfying
	\begin{align*}
		&\nabla_{X}(m_{1} + m_{2}) = \nabla_{X}m_{1} + \nabla_{X}m_{2}\\
		&\nabla_{X}(\lambda m) = \lambda\nabla_{X}m\\
		&\nabla_{X + Y}m = \nabla_{X}m + \nabla_{Y}m\\
		&\nabla_{\lambda X}m = \lambda \nabla_{X}m\\
		&\nabla_{X_{a}}(fm) = \sigma_{a}(f)\nabla_{X_{a}}m + X_{a}(f)\hat{\tau}_{a}(m)
	\end{align*}
	for all $X,Y\in\TSigma$, $m, m_{1}, m_{2} \in M$, $\lambda\in\complex,$ $f \in \A$  and $a\in I$.
\end{definition}

\noindent
Analogously, a right $\sigmatau$-connection on a right $\Sigma$-module
$(M,\{(\sigmah_a,\tauh_a)\}_{a\in I})$ satisfies
\begin{align*}
  \nabla_{X_{a}}(mf) = \sigmah_a(m)X_a(f) + \paraa{\nabla_{X_a}m}\tau_a(f),
\end{align*}
in addition to the linearity properties.  Moreover, if $M$ is a
$\sigmatau$-bimodule and $\nabla$ is a left $\sigmatau$-connection,
as well as a right $\sigmatau$-connection, we say that $\nabla$ is a
$\sigmatau$-bimodule connection (cf. \cite{arnlind2022geometry}).  

We would like to introduce a $\sigmatau$-connection involving
symmetric $\sigmatau$-derivations. Since symmetric
$\sigmatau$-derivations satisfy two Leibniz rules and one thinks of
$\sigmatau$-connections as covariant derivatives in the direction of
the $\sigmatau$-derivation, it is natural to have a notion of
symmetric $\sigmatau$-connection satisfying two Leibniz rules in
correspondence to those of the symmetric $\sigmatau$-derivations. In
the following, we define symmetric $\sigmatau$-connections and discuss
such connections on (strongly) regular and symmetric
$\sigmatau$-algebras.

\begin{definition}
  Let $\Sigma=(\A,\{X_a\}_{a\in I})$ be a $\sigmatau$-algebra and let
  $(M,\{(\sigmah_a,\tauh_a)\}_{a\in I})$ be a left $\Sigma$-module. A
  left \emph{$\sigmatau$-connection on $M$} is called \emph{symmetric} if
  \begin{align}\label{eqn tausigma Leibniz rule for conn}
    \nabla_{X_a}(fm) = \tau_a(f)\nabla_{X_a}m + X_a(f)\sigmah_{a}(m)
  \end{align}
  for all $a\in I$ and $m\in M$.
\end{definition}

\noindent
On $\Sigma$-bimodules, it is easy to construct connections
corresponding to inner $\sigmatau$-derivations.

\begin{proposition}\label{prop:inner.bimodule.connection}
  Let $\Sigma=(\A,\{X_a\}_{a\in I})$ be a $\sigmatau$-algebra such
  that $X_a$ is an inner $(\sigma_a,\tau_a)$-derivation for all $a\in I$; i.e
  there exist $\{f_a\}_{a\in I}\subseteq\A$ such that
  \begin{align*}
    X_a(f) = f_a\tau_a(f) - \sigma_a(f)f_a
  \end{align*}
  for $a\in I$. If $(M,\{(\sigmah_a,\tauh_a)\}_{a\in I})$ is a 
  $\Sigma$-bimodule then
  \begin{align*}
    \nabla_{X_a}(m) = f_a\tauh_{a}(m)-\sigmah_{a}(m)f_a
  \end{align*}
  defines a $\sigmatau$-bimodule connection on $(M,\{(\sigmah_a,\tauh_a)\}_{a\in I})$.
\end{proposition}

\begin{proof}
It is easy to show linearity in $\TSigma$ and $M$ is satisfied.
Let us show  that $\nabla$ satisfies both the left and the right twisted Leibniz rule. For $f \in \A$ and $m \in M,$ one has
\begin{eqnarray*}
&&\nabla_{X_{a}}(fm) = f_{a}\tauh_{a}(fm) - \sigmah_{a}(fm)f_{a} = f_{a}\tau_{a}(f)\tauh_{a}(m) - \sigma_{a}(f)\sigmah_{a}(m)f_{a}\\
&&= \sigma_{a}(f)f_{a}\tauh_{a}(m) -  \sigma_{a}(f)\sigmah_{a}(m)f_{a} + f_{a}\tau_{a}(f)\tauh_{a}(m) - \sigma_{a}(f)f_{a}\tauh_{a}(m)\\
&& = \sigma_{a}(f)(f_{a}\tauh_{a}(m) - \sigmah_{a}(m)f_{a}) + (f_{a}\tau_{a}(f) - \sigma_{a}(f)f_{a})\tauh_{a}(m)\\
&& = \sigma_{a}(f)\nabla_{X_{a}}(m) + X_{a}(f)\tauh_{a}(m),
\end{eqnarray*}
showing that $\nabla$ satisfies the left twisted Leibniz rule. In a similar manner, Let us check that $\nabla$ satisfies the right twisted Leibniz rule. One has
\begin{eqnarray*}
&&\nabla_{X_{a}}(mf) = f_{a}\tauh_{a}(mf) - \sigmah_{a}(mf)f_{a} = f_{a}\tauh_{a}(m)\tau_{a}(f) - \sigmah_{a}(m)\sigma_{a}(f)f_{a}\\
&&= \sigmah_{a}(m)f_{a}\tau_{a}(f) - \sigmah_{a}(m)\sigma_{a}(f)f_{a} + f_{a}\tauh_{a}(m)\tau_{a}(f) - \sigmah_{a}(m)f_{a}\tau_{a}(f)\\
&& = \sigmah_{a}(m)(f_{a}\tau_{a}(f) - \sigma_{a}(f)f_{a}) + (f_{a}\tauh_{a}(m) - \sigmah_{a}(m)f_{a})\tau_{a}(f)\\
&& = \sigmah_{a}(m)X_{a}(f) + \nabla_{X_{a}}(m)\tau_{a}(f).\qedhere
\end{eqnarray*}
\end{proof}

\noindent
Now, let $(\A,\{X_a\}_{a\in I})$ be a symmetric strongly regular
  $\sigmatau$-algebra; it follows from Proposition~\ref{str reg symm
    der} that there exists $\{f_a\}_{a\in I}\subseteq\A$ such that
  $[f_a,\sigma_a(f)]=[f_a,\tau_a(f)]=0$  and
  \begin{align*}
    X_a(f) = f_a\delta_a(f) = f_a\paraa{\tau_a(f)-\sigma_a(f)} 
  \end{align*}
  for $f\in\A$. Given a (left) $\Sigma$-module
  $(M,\{\sigmah_a,\tauh_a\}_{a\in I})$ one can easily check that
  \begin{align*}
    \nabla_{X_a}m = f_a\hat{\delta}_a(m) = f_a\paraa{\tauh_a(m)-\sigmah_a(m)}
  \end{align*}
  defines a symmetric left $\sigmatau$-connection on
  $(M,\{\sigmah_a,\tauh_a\}_{a\in I})$. The following result states
  that this every symmetric $\sigmatau$-connection on a strongly
  regular symmetric $\sigmatau$-algebra is of this form. 

\begin{proposition}\label{gen conn}
  Let $\Sigma=(\A, \{X_{a}\}_{a\in I})$ be a strongly regular
  symmetric $\sigmatau$-algebra and write
  \begin{align*}
    X_a(f) = f_a\delta_a(f) = f_a\paraa{\tau_a(f)-\sigma_a(f)}
  \end{align*}
  for $f_a\in\A$ such that $[f_{a}, \tau_{a}(f)] = [f_{a}, \sigma_{a}(f)] = 0.$ If $\nabla$ is a (left) symmetric
  $(\sigma, \tau)$-connection on a $\Sigma$-module
  $(M,\{\sigmah_a,\tauh_a\}_{a\in I})$ then
  \begin{eqnarray*}
    \nabla_{X_{a}}(m) = f_{a}\hat{\delta}_{a}(m) = f_a\paraa{\tauh(m)-\sigmah(m)}
  \end{eqnarray*}
  for $m\in M$ and $a\in I$.
\end{proposition}

\begin{proof}
  Since $\nabla$ is assumed to be symmetric, one has
  \begin{align*}
    0 &= \nabla_{X_{a}}(fm) - \nabla_{X_{a}}(fm)\\
      &= \sigma_{a}(f)\nabla_{X_{a}}m + X_{a}(f)\tauh_{a}(m) - \tau_{a}(f)\nabla_{X_{a}}m - X_{a}(f)\sigmah_{a}(m) \\
      &= (\sigma_{a}(f) - \tau_{a}(f))\nabla_{X_{a}}m + X_{a}(f)(\tauh_{a}(m) - \sigmah_{a}(m))\\
      &= -\delta_{a}(f)\nabla_{X_{a}}m + X_{a}(f)\deltah_{a}(m).
  \end{align*}
  Using that $X_a(f)=f_a\delta_a(f)$ together with $[f_a,\delta_a(f)]=0$, it follows that
  \begin{align}\label{eq:delta.f.nabla}
    \delta_a(f)\paraa{\nabla_{X_a}m-f_a\deltah_a(m)} = 0
  \end{align}
  for $f\in\A$ and $m\in M$. Since $\sigmatau$ is a strongly regular
  pair, there exists $g_a\in\A$ such that $\delta(g_a)$ is
  invertible. Thus, setting $f=g_a$ in \eqref{eq:delta.f.nabla} gives
  \begin{align*}
    \nabla_{X_a}m = f_a\deltah_a(m)
  \end{align*}
  for $m\in M$.
\end{proof}

\noindent
Let us summarize these results in the following corollary.

\begin{corollary}\label{cor:unique.symmetric.connection}
  Let $\Sigma=(\A, \{X_{a}\}_{a\in I})$ be a strongly regular
  symmetric $\sigmatau$-algebra and let
  $(M,\{\sigmah_a,\tauh_a\}_{a\in I})$ be a left $\Sigma$-module. Then
  there exists a unique symmetric left $\sigmatau$-connection on
  $(M,\{\sigmah_a,\tauh_a\}_{a\in I})$.
\end{corollary}

\section{Curvature}

\noindent
Curvature of $\sigmatau$-connections was defined in
\cite{arnlind2022geometry}. In this Section, we explore the
possibility of curvatures of symmetric $\sigmatau$-connections to
satisfy a twisted linearity condition. In classical differential
geometry, curvature of connections are discussed with reference to the
Lie-algebra structure induced by vector fields on a manifold. In the
case of $\sigmatau$-connections, the curvature is discussed with
reference to the $\sigmatau$-Lie algebra structure. Let us recall a
few definitions.
\begin{definition}
  Let $\Sigma = (\A, \{X_{a}\}_{a \in I})$ be a $\sigmatau$-algebra
  and let $T\Sigma$ denote the complex vector space generated by
  $\{X_{a}\}_{a \in I}.$ Given a $\complex$-bilinear map
  $R: T\Sigma \otimes_{\complex} T\Sigma \to T\Sigma
  \otimes_{\complex} T\Sigma,$ we say that $(T\Sigma, R)$ is a
  $\sigmatau$-\it{Lie algebra} if
  \begin{itemize}
  \item[(1)] $R^{2} = \mbox{id}_{T\Sigma \otimes_{\complex} T\Sigma},$
  \item[(2)] $m(X_{a} \otimes_{\complex} X_{b} - R(X_{a} \otimes_{\complex} X_{b})) \in T\Sigma,$
  \end{itemize}
  for all $a, b \in I,$ where
  $m :T\Sigma \otimes_{\complex} T\Sigma \to \mbox{End}_{\complex}(A)$
  denotes the composition map given by
  $m(X_{a} \otimes_{\complex} X_{b}) =X_{a} \circ X_{b}.$
\end{definition}

\noindent Given a $\sigmatau$-Lie algebra $(\TSigma,R)$ one introduces the
bracket $[\cdot,\cdot]_R:\TSigma\times\TSigma\to\TSigma$ as
\begin{align*}
  [X,Y]_R = m(X_{a} \otimes_{\complex} X_{b} - R(X_{a} \otimes_{\complex} X_{b}))
\end{align*}
for $X,Y\in\TSigma$. Moreover, if $I$ is a finite set, one introduces
the components of $R$, as well as structure constants
$C_{ab}^p\in\complex$, as
\begin{align*}
  [X_a,X_b]_R = R_{ab}^{pq}X_p\otimes X_q = C_{ab}^pX_p
\end{align*}
(where summation over repeated indices is implied).  Now, let us
recall the definition of curvature with respect to a $\sigmatau$-Lie
algebra structure on $\TSigma$.  

\begin{definition}
  Let $\Sigma = (\A, \{X_{a}\}_{a \in I})$ be a $\sigmatau$-algebra,
  such that $(T\Sigma, R)$ is a $\sigmatau$-Lie algebra, and let
  $(M, \{(\sigmah_{a}, \tauh_{a})\}_{a \in I})$ be a
  $\Sigma$-module. Given a $\sigmatau$-connection $\nabla,$ the
  curvature of $\nabla$ is defined as
  \begin{eqnarray*}
    \mbox{Curv}(X, Y)n = m_{\nabla}(X \otimes_{\complex} Y - R(X \otimes_{\complex} Y))n - \nabla_{[X, Y]_{R}}n
  \end{eqnarray*}
  for $n \in M$ and $X, Y \in T\Sigma,$ where
  $m_{\nabla}(X \otimes_{\complex} Y) = \nabla_{X} \circ \nabla_{Y}.$
\end{definition}

  \noindent In components, the curvature can be written as 
\begin{eqnarray*}
	\mbox{Curv}(X_{a}, X_{b}) = \nabla_{X_{a}}  \circ \nabla_{X_{b}} - R_{ab}^{pq} \nabla_{X_{p}} \circ \nabla_{X_{q}} - C_{ab}^{p}\nabla_{X_{p}}.
\end{eqnarray*}

\noindent
Let us now discuss the linearity properties of curvature of
$\sigmatau$-connections. In general, the curvature of
$\sigmatau$-connections does not satisfy any linearity property. The
linearity property of curvature can be satisfied depending on how the
map $R$ on the $\sigmatau$-Lie algebra is defined, how the
endomorphisms $\sigma_{a}$'s and $\tau_{a}$'s on the algebra interact
and how the maps $\tauh_{a}$'s on the $\Sigma$-module interact with
the connection. For example, for $f \in \A$ and $m \in M,$ one has
\begin{align*}
  \curv(X_{a}, X_{b})(fm)
  &= \nabla_{X_{a}}(\nabla_{X_{b}}(fm)) - R_{ab}^{pq}\nabla_{X_{p}}(\nabla_{X_{q}}(fm))
    - C_{ab}^{p}\nabla_{X_{p}}(fm)\\
  &= \nabla_{X_{a}}(\sigma_{b}(f)\nabla_{X_{b}}(m) + X_{b}(f)\tauh_{b}(m))
    - R_{ab}^{pq}\nabla_{X_{p}}(\sigma_{q}(f)\nabla_{X_{q}}(m)\\
  &\qquad+ X_{q}(f)\tauh_{q}(m))
  - C_{ab}^{p}(\sigma_{p}(f)\nabla_{X_{p}}(m) + X_{p}(f)\tauh_{p}(m))\\
  & = \sigma_{a}(\sigma_{b}(f))\nabla_{X_{a}}(\nabla_{X_{b}}(m))
    + X_{a}(\sigma_{b}(f))\tauh_{a}(\nabla_{X_{b}}(m))\\
  &\qquad + \sigma_{a}(X_{b}(f))\nabla_{X_{a}}(\tauh_{a}(m)) + X_{a}(X_{b}(f))\tauh_{a}(\tauh_{b}(m))\\
  &\qquad - R_{ab}^{pq}(\sigma_{p}(\sigma_{q}(f))\nabla_{X_{p}}(\nabla_{X_{q}}(m)) + X_{p}(\sigma_{q}(f))\tauh_{p}(\nabla_{X_{q}}(m)))\\
  &\qquad - R_{ab}^{pq}(\sigma_{p}(X_{q}(f)) + X_{p}(X_{q}(f))\tauh_{p}(\tauh_{q}(m)))\\
  &\qquad - C_{ab}^{p}(\sigma_{p}(f)\nabla_{X_{p}}(m) + X_{p}(f)\tauh_{p}(m)).
\end{align*}
Despite that the curvature of a general $\sigmatau$-connection potrays no linearity property in general, for specific $\sigmatau$-Lie algebra or specific $\sigmatau$-connections, the curvature can satisfy a linearity property. For example in \cite{arnlind2022geometry},  the curvature of a $\sigmatau$-connection on a matrix algebra satisfies a certain twisted linearity property. To this effect, let us consider curvature of a general $\sigmatau$-connection with reference to a specific $\sigmatau$-Lie algebra. Assume that 
\begin{eqnarray*}
	R_{ab}^{pq} = \delta_{a}^{q}\delta_{b}^{p} = \left\{\begin{array}{ll} 
		1 & \mbox {if}\quad a = q, b = p,\\ 
		0 & \mbox{otherwise,}
	\end{array}
	\right.\\
\end{eqnarray*}
and  $C_{ab}^{p}= 0$ for all $a, b \in I.$ Then the bracket on the generators of $(\TSigma, R)$ is 
\begin{eqnarray}\label{eqn comm rel curv}
	[X_{a}, X_{b}]_{R} = X_{a}\circ X_{b} - X_{b}\circ X_{a} = 0,
\end{eqnarray}
implying that the curvature of a $\sigmatau$-connection $\nabla$ can now be written as
  \begin{eqnarray*}
	\mbox{Curv}(X_{a}, X_{b}) = \nabla_{X_{a}}\circ\nabla_{X_{b}} - \nabla_{X_{b}}\circ\nabla_{X_{a}}.
\end{eqnarray*}

\begin{proposition}\label{prop twisted lin curv}
	Let $\Sigma = (\A, \{X_{a}\}_{a\in I})$ be a $\sigmatau$-algebra such that $(\TSigma, R)$ is the $\sigmatau$-Lie algebra with bracket
	\begin{eqnarray*}
		[X_{a}, X_{b}]_{R} = X_{a} \circ X_{b} - X_{b}\circ X_{a} = 0
	\end{eqnarray*}
and let $(M, \{(\sigmah_{a}, \tauh_{a})\}_{a\in I})$ be a left  $\Sigma$-module. If $\nabla:\TSigma\times M\to M$ is a left $\sigmatau$-connection and 
\begin{eqnarray}\label{comm rel twist lin cond}
	&&[\sigma_{a}, \sigma_{b}] = 0 = [\sigma_{a}, X_{b}],\\ \nonumber
	&&[\tauh_{a}, \tauh_{b}] = 0 = [\tauh_{a}, \nabla_{X_{b}}]
\end{eqnarray}
for all $a, b \in I,$ then
\begin{eqnarray}\label{eqn twist linear curv}
	\curv(X_{a}, X_{b}(fm)) = \sigma_{a}(\sigma_{b}(f))\curv(X_{a}, X_{b})m
\end{eqnarray}
for $f \in \A$ and $m\in M.$
\end{proposition}

\begin{proof}
Since the $\sigmatau$-Lie algebra bracket is given by \eqref{eqn comm rel curv}, the curvature is written as 
\begin{eqnarray*}
	\curv(X_{a}, X_{b}) = \nabla_{X_{a}}\circ\nabla_{X_{b}} - \nabla_{X_{b}}\circ\nabla_{X_{a}}.
\end{eqnarray*}
Now, for $f\in \A$ and $m \in M,$ and applying the commutation relations \eqref{comm rel twist lin cond}, one has
\begin{eqnarray*}\nonumber 
	&&\curv(X_{a}, X_{b})(fm) = \nabla_{X_{a}}(\nabla_{X_{b}}(fm)) - \nabla_{X_{b}}(\nabla_{X_{a}}(fm))\\ \nonumber
	&&= \nabla_{X_{a}}(\sigma_{b}(f)\nabla_{X_{b}}(m) + X_{b}(f)\tauh_{b}(m)) - \nabla_{X_{b}}(\sigma_{a}(f)\nabla_{X_{a}}(m) + X_{a}(f)\tauh_{a}(m))\\ \nonumber
	&&= \sigma_{a}(\sigma_{b}(f))\nabla_{X_{a}}(\nabla_{X_{b}}(m)) + X_{a}(\sigma_{b}(f))\tauh_{a}(\nabla_{X_{b}}(m)) + \sigma_{a}(X_{b}(f))\nabla_{X_{a}}(\tauh_{b}(m))\\ \nonumber
	&& + X_{a}(X_{b}(f))\tauh_{a}(\tauh_{b}(m))\\
	&& - \sigma_{b}(\sigma_{a}(f))\nabla_{X_{b}}(\nabla_{X_{a}}(m)) - X_{b}(\sigma_{a}(f))\tauh_{b}(\nabla_{X_{a}}(m)) - \sigma_{b}(X_{a}(f))\nabla_{X_{b}}(\tauh_{a}(m))\\ \nonumber
	&& - X_{b}(X_{a}(f))\tauh_{b}(\tauh_{a}(m))\\ \nonumber
	&&=  \sigma_{a}(\sigma_{b}(f))(\nabla_{X_{a}}(\nabla_{X_{b}}(m))
	 -\nabla_{X_{b}}(\nabla_{X_{a}}(m))) \\ \nonumber
	&&= (\sigma_{a}\circ \sigma_{b})(f))\mbox{Curv}(X_{a}, X_{b})m,
	\end{eqnarray*}
showing the twisted linearity property.
\end{proof}

\noindent
From the Proposition above, not all left $\sigmatau$-connections have
a curvature that satisfies the twisted linearity property \eqref{eqn
  twist linear curv}. The commutation relations \eqref{comm rel twist
  lin cond} of Proposition \ref{prop twisted lin curv} restricts the
left $\sigmatau$-connections to be considered due to that the
connection is required to commute with the maps $\tauh_{a}$ for all
$a,$ i.e., $[\tauh_{a}, \nabla_{X_{b}}] = 0$ for all $a, b.$ In this
Proposition, we would like curvature to satisfy the twisted linearity
property \eqref{eqn twist linear curv} without much restrictions on
the left $\sigmatau$-connection. However, this may impose additional
conditions on the endomorphisms and the $\sigmatau$-derivations. Let
us take for example, the left $\sigmatau$-connection given by
$\nabla_{X_{a}}(e_{i}) = \gamma_{ai}^{k}e_{k}$ with
$\gamma_{ai}^{k} \in \A$, defined on the free $\Sigma$-module of
Example \ref{free Sig mod}. One has
\begin{eqnarray*}
	&&[\tauh_{a}, \nabla_{X_{b}}](m^{i}e_{i}) = \tauh_{a}(\nabla_{X_{b}}(m^{i}e_{i})) - \nabla_{X_{b}}(\tau_{a}(m^{i})e_{i})\\ 
	&& = \tau_{a}(\sigma_{b}(m^{i}))\tau_{a}(\gamma_{bi}^{k})e_{k} + \tau_{a}(X_{b}(m^{i}))e_{i} - \sigma_{b}(\tau_{a}(m^{i}))\gamma_{bi}^{k}e_{k} - X_{b}(\tau_{a}(m^{i}))e_{i}\\
	&&=(\tau_{a}(\sigma_{b}(m^{i}))\tau_{a}(\gamma_{bi}^{k}) - \sigma_{b}(\tau_{a}(m^{i}))\gamma_{bi}^{k} + \tau_{a}(X_{b}(m^{k})) - X_{b}(\tau_{a}(m^{k})))e_{k},
\end{eqnarray*} 
which does not equal zero naturally without imposing further
conditions. In this example, one may choose the additional condition
$[\tau_{a}, \sigma_{b}] = 0 = [\tau_{a}, X_{b}]$ and set
$\tau_{a}(\gamma_{bi}^{k}) = \gamma_{bi}^{k},$ for all $a, b.$ To
stick to the conditions in Proposition \ref{prop twisted lin curv} and
to the additional conditions imposed by the example above, it is
natural to choose $\tau_{a} = \mbox{Id}$ to get a curvature that
satisfies the twisted linearity property. In the following Corollary,
the curvature of all left $(\sigma, \mbox{Id})$-connections for the
said $\sigmatau$-Lie algebra satisfies the twisted linearity property.
 
\begin{corollary}\label{cor twisted lin curv}
  Let $\Sigma = (\A, \{X_{a}\}_{a\in I}$ be a $\sigmatau$-algebra such
  that $(\TSigma, R)$ is a $\sigmatau$-Lie algebra with bracket
  \begin{eqnarray*}
    [X_{a}, X_{b}]_{R} = X_{a}\circ X_{b} - X_{b}\circ X_{a} = 0
  \end{eqnarray*}
  such that $X_{a}$ is a $(\sigma_{a}, \mbox{Id})$-derivation satisfying 
  \begin{eqnarray}\label{cor comm rel twist lin cond}
    [\sigma_{a}, \sigma_{b}] = 0 = [\sigma_{a}, X_{b}]
  \end{eqnarray}
  for $a, b \in I.$ Then it follows that the curvature of all left
  $(\sigma, \mbox{Id})$-connections on the $\Sigma$-module
  $(M, \{(\sigmah_{a}, \mbox{Id})\}_{a\in I})$ satisfy the twisted
  linearity property
  \begin{eqnarray*}
    \curv(X_{a}, X_{b})(fm) = \sigma_{a}\paraa{\sigma_{b}(f)}\curv(X_{a}, X_{b})m.
  \end{eqnarray*}
  for $f\in\A$ and $m\in M$
\end{corollary}

\noindent
 Let us now discuss the curvature of the unique symmetric
 $\sigmatau$-connection of Proposition \ref{gen conn}. Recall that the
 unique connection $\nabla_{X_{a}}(m) = f_{a}\deltah_{a}(m)$ is
 defined on a strongly regular symmetric $\sigmatau$-algebra
 $(\A, \{X_{a}\}_{a\in I})$ where $X_{a} = f_{a}\delta_{a}$ are inner
 $\sigmatau$-derivations for $a \in I.$ Choosing
 $\sigmah_{a} = \tauh_{a} = \mbox{Id}$ leaves the connection
 trivial. We would like to focus on the special case when
 $\sigmah_{a} \neq \tauh_{a} = \mbox{Id}$ and $f_{a} \in \complex,$
 for $a \in I.$ To this end, let $\Sigma = (\A, \{X_{a}\}_{a\in I})$ be the
 strongly regular symmetric $\sigmatau$-algebra where
 $X_{a} = f_{a}\delta_{a} = f_{a}(\mbox{Id} - \sigma_{a})$ for
 $f_{a} \in \complex$ and $a \in I.$
In this case the vector space $\TSigma$ is spanned by $\{\delta_a\}_{a\in I}$.
 Define
 $R:\TSigma \otimes_{\complex} \TSigma \to \TSigma \otimes_{\complex}
 \TSigma$ by
 \begin{eqnarray*}
   R(X_{a}\otimes_{\complex} X_{b}) = R(f_{a}\delta_{a} \otimes_{\complex} f_{b}\delta_{b}) = f_{b}f_{a}(\delta_{b} \otimes_{\complex}\delta_{a})
 \end{eqnarray*}
 implying that $R^{2} = \mbox{Id}.$ Assuming that
 $[\sigma_{a}, \sigma_{b}] = 0,$ one has
 \begin{eqnarray*}
   [X_{a}, X_{b}]_{R} = m(X_{a} \otimes_{\complex} X_{b} - R(X_{a} \otimes_{\complex} X_{b})) = f_{a}f_{b}(\delta_{a} \circ \delta_{b} - \delta_{b} \circ \delta_{a}) = 0,
 \end{eqnarray*}
 showing that the bracket, $[X_{a}, X_{b}]_{R} = 0,$ on the
 $\sigmatau$-Lie algebra $(\TSigma, R).$ Let us compute the curvature
 of the unique connection for the $\sigmatau$-algebra described
 above. Since $[X_{a}, X_{b}]_{R} = 0,$ then
 $\nabla_{[X_{a}, X_{b}]_{R}]}(m) = 0$ for $m\in M$ and
 $X_{a}, X_{b} \in \TSigma,$ and the curvature of the unique
 connection is
\begin{eqnarray*}
\curv(X_{a}, X_{b})m = f_{a}f_{b}(\deltah_{a} \circ \deltah_{b} - \deltah_{b}\circ \deltah_{a})(m).
\end{eqnarray*}
We remark that the curvature above is not zero in general because
$[\sigma_{a}, \sigma_{b}] = 0$ does not imply
$[\sigmah_{a}, \sigmah_{b}] = 0.$ Given Corollary \ref{cor twisted lin
  curv}, one checks that the commutation relations \eqref{cor comm rel
  twist lin cond} are satisfied on the $\sigmatau$-algebra and since
the bracket $[X_{a}, X_{b}]_{R} = 0,$ one concludes that the unique
$(\sigma, \mbox{Id})$-connection satisfies the twisted linearity
condition \eqref{eqn twist linear curv}.

Let us now review the example of a $\sigmatau$-connection over the
matrix algebra $\mat$, constructed in \cite{arnlind2022geometry}, and
relate it to symmetric $\sigmatau$-connections.

Consider the matrix algebra $\mat$ and choose invertible
$(N\times N)$-matrices $U_{a},$ for $ a = 1, \ldots, n$ such that
$[U_{a}, U_{b}] = 0$ for all $a, b.$ Define endomorphisms
$\sigma_{a}:\mat\to\mat$ by
  \begin{eqnarray*}
    \sigma_{a}(A) = U_{a}AU_{a}^{-1}
  \end{eqnarray*}
  and symmetric (inner) $(\sigma_{a}, \mbox{Id})$-derivations
  $X_{a}:\mat\to\mat$ by
  \begin{eqnarray*}
    X_{a}(A) = A - U_{a}AU_{a}^{-1}
  \end{eqnarray*}
  for $A\in \mat$ and $a= 1, \ldots, n.$ As described in
  \cite{arnlind2022geometry}, $\Sigma = (\mat, \{X_{a}\}_{a = 1}^{n})$
  is a $\sigmatau$-algebra and $(\TSigma, R)$, with
  $R(X\otimes Y)=Y\otimes X$, is a $\sigmatau$-Lie algebra such that
  \begin{align*}
    (X_{a} \circ X_{b})(A)
    &= X_{a}(X_{b}(A)) = X_{a}(A - U_{b}AU_{b}^{-1})\\
    & = (A - U_{b}AU_{b}^{-1}) - U_{a}(A - U_{b}AU_{b}^{-1})U_{a}^{-1} \\
    & = (A - U_{a}AU_{a}^{-1}) - U_{b}(A - U_{a}AU_{a}^{-1})U_{b}^{-1}\\
    &= X_{b}(A - U_{a}AU_{a}^{-1})= (X_{b} \circ X_{a})(A)
  \end{align*}
  showing that $[X_{a}, X_{b}]_{R} = 0.$

  Since $\mat$ is a free module over itself and generated by the
  identity matrix $\mathds{1}$, then
  $(\mat, \{(\sigma_{a}, \mbox{Id})\}_{a=1}^{n})$ is a free
  $\Sigma$-module. On this $\Sigma$-module, one can construct a
  $\sigmatau$-connection by setting
  $\nabla_{X_{a}}(\mathds{1}) = \Gammat_{a}$ and
  $\Gamma = \mathds{1} - \Gammat_{a}$ where $\Gammat_{a}\in\mat$
  for $a = 1, \ldots, n.$ This connection is explicitly given by
\begin{eqnarray}\label{sigmatau conn on mat}
	\nabla_{X_{a}}(A) = A - U_{a}AU_{a}^{-1}\Gamma_{a}.
\end{eqnarray}
Let us show that the left $(\sigma, \mbox{Id})$-connection
\eqref{sigmatau conn on mat} is a symmetric left
$\sigmatau$-connection.  Let us take the same $\sigmatau$-algebra
$\Sigma = (\mat, \{X_{a}\}_{a\in I})$ and the same free module $\mat.$
Now define the maps $\sigmah_{a} : \mat\to\mat$ by
$\sigmah_{a}(A) = U_{a}AU_{a}^{-1}\Gamma_{a}$ for $a = 1, \ldots, n.$ Then
$(\mat, \{(\sigmah_{a}, \mbox{Id})\}_{a =1}^{n})$ is a left
$\Sigma$-module. Define the $(\sigma, \mbox{Id})$-connection
\eqref{sigmatau conn on mat}, on this new left $\Sigma$-module
$(\mat, \{(\sigmah_{a}, \mbox{Id})\}_{a =1}^{n}).$ Since
\eqref{sigmatau conn on mat} is a left
$(\sigma, \mbox{Id})$-connection, it suffices to show that it
satisfies \eqref{eqn tausigma Leibniz rule for conn}. Let
$A, B \in \mat,$ one has
\begin{eqnarray*}
&&\nabla_{X_{a}}(AB) = AB - U_{a}ABU_{a}^{-1}\Gamma_{a}\\
&& = AB - AU_{a}BU_{a}^{-1}\Gamma_{a} + AU_{a}BU_{a}^{-1}\Gamma_{a} - U_{a}AU_{a}^{-1}U_{a}BU_{a}^{-1}\Gamma_{a}	\\
&&= A(B - U_{a}BU_{a}^{-1}\Gamma_{a}) + (A - U_{a}AU_{a}^{-1})U_{a}BU_{a}^{-1}\Gamma_{a}\\
&& = A\nabla_{X_{a}}(B) + X_{a}(A)\sigmah_{a}(B),
\end{eqnarray*}
showing that \eqref{sigmatau conn on mat} is a symmetric left
$\sigmatau$-connection on $(\mat,\{\sigmah_a,\Id\}_{a=1}^n)$.  It is shown in \cite{arnlind2022geometry}
that the curvature of this connection satisfies the twisted linearity
condition \eqref{eqn twist linear curv}. Since the bracket on
$(\TSigma, R)$ is zero, i.e., $[X_{a}, X_{b}]_{R} = 0$ for
$X_{a}, X_{b} \in \TSigma,$ the curvature of any
$\sigmatau$-connection on
$(\mat, \{(\sigmah_{a}, \mbox{Id})\}_{a=1}^{n})$ is given as
\begin{eqnarray*}
	\curv(X_{a}, X_{b}) = \nabla_{X_{a}} \circ \nabla_{X_{b}} - \nabla_{X_{b}}\circ \nabla_{X_{a}}.
\end{eqnarray*}
Explicitly, the curvature of the $(\sigma, \mbox{Id})$-connection \eqref{sigmatau conn on mat} is 
\begin{eqnarray*}
	\curv(X_{a}, X_{b})A = U_{a}U_{b}A[U_{b}^{-1}\Gamma_{b}, U_{a}^{-1}\Gamma_{a}]
\end{eqnarray*}
for $A\in\mat.$ We would like to verify that the commutation relations
\eqref{comm rel twist lin cond} is satisfied in this case. Since
$\tauh_{a} = \tau_{a} = \mbox{Id},$ the two commutation relations
$[\tauh_{a}, \tauh_{b}] = 0$ and $[\tauh_{a}, \nabla_{X_{b}}] = 0$ are
already satisfied. One needs to verify $[\sigma_{a}, \sigma_{b}] = 0$
and $[\sigma_{a}, X_{b}] = 0.$ Let $A \in \mat,$ one has
\begin{align*}
&	[\sigma_{a}, \sigma_{b}](A) = \sigma_{a}(\sigma_{b}(A)) - \sigma_{b}(\sigma_{a}(A))  = \sigma_{a}(U_{b}AU_{b}^{-1}) - \sigma_{b}(U_{a}AU_{a}^{-1})\\
& = U_{a}U_{b}AU_{b}^{-1}U_{a}^{-1} - U_{b}U_{a}AU_{a}^{-1}U_{b}^{-1} = 0,
\end{align*}
and
\begin{align*}
&[\sigma_{a}, X_{b}](A) = \sigma_{a}(X_{b}(A)) - X_{b}(\sigma_{a}(A)) = \sigma_{a}(A - U_{b}AU_{b}^{-1}) - X_{b}(U_{a}AU_{a}^{-1})\\
&= U_{a}(A -U_{b}AU_{b}^{-1})U_{a}^{-1} - U_{a}AU_{a}^{-1} + U_{b}(U_{a}AU_{a}^{-1})U_{b}^{-1}\\
& = U_{a}AU_{a}^{-1} - U_{a}U_{b}AU_{b}^{-1}U_{a}^{-1} - U_{a}AU_{a}^{-1} + U_{b}U_{a}AU_{a}^{-1}U_{b}^{-1} = 0,
\end{align*}
(using that $[U_a,U_b]=0$) showing that the commutation relations
\eqref{comm rel twist lin cond} are satisfied. Therefore, by Corollary
\ref{cor twisted lin curv}, we have verified that the left symmetric
$(\sigma, \mbox{Id})$-connection \eqref{sigmatau conn on mat} defined
on the $\Sigma$-module
$(\mat, \{(\sigmah_{a}, \mbox{Id})\}_{a =1}^{n})$ such that the
relations \eqref{comm rel twist lin cond} are satisfied and the
bracket of the $\sigmatau$-Lie algebra $(T\Sigma, R)$ satisfies
\eqref{eqn comm rel curv} has curvature satisfying the twisted
linearity property \eqref{eqn twist linear curv}.

\section{Commutative $\sigmatau$-algebras}

\noindent
In this Section, we demonstrate that the notion of regular and
strongly regular pairs on a commutative algebra can be used to
construct modules of $\sigmatau$-derivations on which curvature and
torsion of some connections can be computed.  Note that when $\A$ is
commutative, $\mbox{Der}_{\sigmatau}(\A)$ is an $\A$-module and it
consist of symmetric $\sigmatau$-derivations. For example, on a unique
factorisation domain $\A,$ the $\A$-module
$\mbox{Der}_{\sigmatau}(\A)$ is free and of rank one with generator
$\frac{\delta}{g}$ where $g$ is the greatest common divisor of
the image of $\delta$ \cite{hls:sigmaderivation}.
It is further shown
in \cite{hls:sigmaderivation} that every $\sigmatau$-derivation is of
the form $ X(f) = f_{0}\frac{\delta(f)}{g}.$ Let us recall these results.

\begin{theorem}[\cite{hls:sigmaderivation}]\label{theo HLS}
  Let $\A$ be a unique factorization domain and let
    $\sigma,\tau\in\End(\A)$ such that $\sigma\neq\tau$. Then
  $\Der_{\sigmatau}(\A)$ is a free $\A$-module of rank one generated by
  \begin{equation}\label{eq:Delta.generator}
    \Delta = \dfrac{(\tau - \sigma)}{g}: x\longmapsto \dfrac{(\tau - \sigma)(x)}{g}
  \end{equation}
  where $g = \mbox{gcd}((\tau - \sigma)(A)).$
\end{theorem}

\noindent
Note that the $\sigmatau$-derivation $\Delta$ in
\eqref{eq:Delta.generator} is not necessarily an inner
$\sigmatau$-derivation when $g$ is not invertible in $\A$. Now, let us
try to understand what happens on a commutative algebra that is not a
unique factorization domain. Let us start by showing that in case
there is a common divisor, then a $\sigmatau$-derivation has to have
a similar form as in \eqref{eq:Delta.generator}.

\begin{proposition}\label{prop X der form}
  Let $X$ be a $\sigmatau$-derivation on a commutative algebra $\A.$
  If there exists $g_{0} \in \A$ such that $\delta(g_{0})$ divides
  $\delta(f)X(g_{0})$ for all $f \in \A,$ then
  \begin{eqnarray*}
    X(f) = \dfrac{X(g_{0})\delta(f)}{\delta(g_{0})}.
  \end{eqnarray*}
\end{proposition}

\begin{proof}
Since $X$ is a symmetric $\sigmatau$-derivation, Lemma \ref{lemma:Xd.dX} implies that
\begin{eqnarray*}
 X(f)\delta(g) = \delta(f)X(g)	
\end{eqnarray*}
for $f, g \in \A.$ Now, let $g_{0} \in \A$ be such that $\delta(g_{0})$ divides $X(g_{0})\delta(f)$ for all $f\in\A,$ one has
\begin{eqnarray*}
	\dfrac{X(f)\delta(g_{0})}{\delta(g_{0})} = \dfrac{X(g_{0})\delta(f)}{\delta(g_{0})} \implies X(f) = \dfrac{X(g_{0})\delta(f)}{\delta(g_{0})},
\end{eqnarray*}
as required.
\end{proof} 

  \noindent The existence of a common divisor is in fact a
  characterization of inner $\sigmatau$-derivations, as the next
  result shows.  

\begin{corollary}\label{cor X der form}
  Let $X$ be a $\sigmatau$-derivation on a commutative algebra $\A.$
  Then $X$ is inner if and only if there exist $g_{0}\in \A$ such that
  $\delta(g_{0})$ divides $X(g_{0}).$
\end{corollary}

\begin{proof}
  Assume that there exist $g_{0} \in \A$ such that $\delta(g_{0})$
  divides $X(g_{0}),$ then $\delta(g_{0})$ divides $X(g_{0})\delta(f)$
  for all $f \in \A$ and by Proposition \ref{prop X der form},
  \begin{eqnarray*}
    X(f) = \dfrac{X(g_{0})\delta(f)}{\delta(g_{0})}.
  \end{eqnarray*}
  Since $\delta(g_{0})$ divides $X(g_{0})$ let $P_{0}$ be such that
  $X(g_{0}) = P_{0} \delta(g_{0}),$ one has
  \begin{eqnarray*}
    X(f) = 	\dfrac{X(g_{0})\delta(f)}{\delta(g_{0})} = \dfrac{P_{0} \delta(g_{0})\delta(f)}{\delta(g_{0})} = P_{0} \delta(f),
  \end{eqnarray*}
  showing that $X$ is inner. Conversely, assume that $X$ is inner. Then, by definition,
  $\delta(f)$ divides $X(f)$ for all $f\in \A.$ Thus there exist
  $g_{0} \in \A$ such that $\delta(g_{0})$ divides $X(g_{0}).$
\end{proof}

\noindent
Let us look at how Proposition \ref{prop X der form} applies to
$\sigmatau$-derivations on the polynomial algebra $\mathbb{K}[x]$
(which is a unique factorization domain, implying that
Theorem~\ref{theo HLS} may also be applied).  Assume that
$\sigma(p)(x) = p(qx)$ and $\tau(p)(x) = p(x)$ for
$p(x) \in \mathbb{K}[x]$ and $q \in \reals\backslash \{0\}.$ One has
\begin{eqnarray*}
  &&\delta(p) = \tau \left(\sum_{i= 0}^{N}a_{i}x^{i}\right) - \sigma\left(\sum_{i= 0}^{N}a_{i}x^{i}\right) = \sum_{i=1}^{N}(1 -q^{i})a_{i}x^{i} \\
  &&= \sum_{k = 0}^{N-1}\paraa{(1+q+\cdots+q^k)a_{k + 1}x^{k}}(1-q)x = r(x)(1-q)x,
\end{eqnarray*}
implying that $(1 - q)x$ is the greatest common divisor of
$\delta(\mathbb{K}[x]).$ From Proposition~\ref{prop X der form},
setting $g_{0} = x,$ one has $\delta(x) = (1 - q)x$ and it follows
that any $\sigmatau$-derivation $X$ on $\mathbb{K}[x]$ is of the form
\begin{equation}\label{eq:X.px}
  X = p(x)\dfrac{\delta}{(1 - q)x}
\end{equation}
for $p(x) = \mathbb{K}[x]$ determined by $X(x).$ From
  Corollary~\ref{cor X der form} it follows that the derivation $X$
  above is inner if an only if $p(x)$ is divisible by $x$, that is if
  $p(x)$ has no constant term. In other words, if $p(x)$ has a
  constant term, then the derivation $X$ in \eqref{eq:X.px} is an
  outer derivation. Furthermore, this implies that the Jackson
derivative, defined in \eqref{Jackson der}, is an outer
$\sigmatau$-derivation on $\mathbb{K}[x]$ since $p(x)=1$ in that
case.

Let us now consider a commutative algebra $\A$ (not necessarily a
  unique factorization domain) and let $\hat{g}\in \A$ be a common
  divisor of $\delta(\A)$; i.e. $\hat{g}$ divides $\delta(f)$ for all
  $f\in\A$. For each $f\in \A,$ define a linear map
$X_{f}: \A\to\A$ by
\begin{eqnarray*}
  X_{f}(a) = f\frac{\delta(a)}{\hat{g}}
\end{eqnarray*}
for all $a \in \A.$ One has
\begin{eqnarray*}
  && X_{f}(ab) = f\frac{\delta(ab)}{\hat{g}} = f\frac{\sigma(a)\delta(b) + \delta(a)\tau(b)}{\hat{g}} = \sigma(a)f\frac{\delta(b)}{\hat{g}} +  f\frac{\delta(a)}{\hat{g}}\tau(b)\\
  && = \sigma(a)X_{f}(b) + X_{f}(a)\tau(b),
 \end{eqnarray*}
 showing that $X_{f}$ is a $\sigmatau$-derivation. Since $\A$ is
 commutative, $X_{f}$ is symmetric. In analogy with Theorem~\ref{theo
   HLS}, let us show that if the common divisor actually lies in the
 image of $\delta$ and $\sigmatau$ is a regular pair, then every
 $\sigmatau$-derivation is of this form.

\begin{proposition}
  Let $\sigmatau$ be a regular pair of endomorphisms on a commutative
  algebra $\A$. If $\hat{g}$ is a common divisor of $\Im(\delta)$ such
  that $\hat{g}\in\Im(\delta)$ then $\mbox{Der}_{\sigmatau}(\A)$ is a
  free module of rank one with basis $\delta/\hat{g}$.
\end{proposition}

\begin{proof}
  Let $g_0\in\A$ such that $\delta(g_0)=\hat{g}$. Lemma~\ref{lemma:Xd.dX} implies that
  \begin{align*}
    X(f)\hat{g}=\delta(f)X(g_0)\implies
    X(f) = X(g_0)\frac{\delta(f)}{\hat{g}} 
  \end{align*}
  since $\hat{g}$ is a common divisor of the image of $\delta$, which
  shows that $\delta/\hat{g}$ generates
  $\mbox{Der}_{\sigmatau}(\A)$. Now, assuming $f_0\delta(f)/\hat{g}=0$
  for all $f\in\A$ implies that $f_0\delta(f)=0$ for all $f\in\A$
  which gives $f_0=0$ since $\sigmatau$ is assumed to be a regular
  pair. Hence, $\delta/\hat{g}$ is a basis for $\mbox{Der}_{\sigmatau}(\A)$.
\end{proof}

 \noindent
 The set
 $M_{\hat{g}} = \{f\frac{\delta}{\hat{g}}: f\in\A\}$ is a collection
 of symmetric $\sigmatau$-derivations of $\A$ and the pair
 $(\A, M_{\hat{g}})$ is a symmetric $\sigmatau$-algebra. Let
 $\Sigma_{\hat{g}} = (\A, M_{\hat{g}})$ and let $\TSigma_{\hat{g}}$ be the
 tangent space generated by $M_{\hat{g}}$.
 It is easy to see that $\TSigma_{\hat{g}}=M_{\hat{g}}$ since
 \begin{align*}
   X_f + X_g = (f+g)\frac{\delta}{\hat{g}} = X_{f+g}.
 \end{align*}
 
\noindent
Moreover, the vector space $\TSigma_{\hat{g}}$ can be endowed with an $A$-module structure by defining 
\begin{equation*}
  f\cdot X_{g} = X_{fg}
\end{equation*}
for $f, g \in \A.$ Indeed, if $f_{1}, f_{2} \in \A$ and $X_{g} \in \TSigma_{\hat{g}},$ then one has 
\begin{eqnarray*}
	(f_{1} + f_{2}) \cdot X_{g} = X_{(f_{1} + f_{2})g} = X_{f_{1}g + f_{2}g} = X_{f_{1}g} + X_{f_{2}g} = f_1\cdot X_{g} + f_{2}\cdot X_{g}
\end{eqnarray*}
and 
\begin{eqnarray*}
	f_{1}\cdot(f_{2}\cdot X_{g}) = f_{1}\cdot X_{f_{2}g} = X_{f_{1}(f_{2}g)} = X_{(f_{1}f_{2})g} = (f_{1}f_{2})\cdot X_{g},
\end{eqnarray*}
as well as
\begin{eqnarray*}
	f\cdot(X_{g_{1}} + X_{g_{2}}) = f\cdot X_{g_{1} + g_{2}} = X_{fg_{1} + fg_{2}} = f\cdot X_{g_{1}} + f\cdot X_{g_{2}}.
\end{eqnarray*} 
for $g_1,g_2\in\A$.

Thus, $\TSigma_{\hat{g}}$ is an $\A$-module and it consists of
$\sigmatau$-derivations of $\A$ and it is
a submodule of the module of
$\sigmatau$-derivations $\mbox{Der}_{\sigmatau}(\A)$.
The next result shows that if $\sigmatau$ is a regular pair, then
$\delta/\hat{g}$ is a basis for $\TSigma_{\hat{g}}$.

\begin{proposition}\label{prop regular}
  Let $\sigmatau$ be a regular pair of endomorphisms on a commutative
  algebra $\A$ and let $\Sigma_{\hat{g}} = (\A, M_{\hat{g}})$ where
  $\hat{g}$ is a common divisor of $\delta(\A).$ Then
  $\frac{\delta}{\hat{g}}$ is a basis of $\TSigma_{\hat{g}}.$
\end{proposition}

\begin{proof}
  We would like to show that if $f\frac{\delta}{\hat{g}} = 0$ then
  $f = 0$. Since $\sigmatau$ is a regular pair, there exist
  $a_{0}\in \A$ such that $\delta(a_{0})$ is not a zero divisor. Let
  us first show that if $\delta(a_{0})$ is not a zero divisor then
  $\frac{\delta(a_{0})}{\hat{g}}$ is not a zero divisor. For the sake
  of argument, assume that $\frac{\delta(a_0)}{\hat{g}}$ is a zero
  divisor. Then there exist $f_{0}\in\A$ such that
  $f_{0}\frac{\delta(a_0)}{\hat{g}} = 0$, which implies
  \begin{eqnarray}\label{zero div}
    \hat{g}f_{0}\frac{\delta(a_0)}{\hat{g}} = 0 \implies f_{0}\delta(a_0) = 0
  \end{eqnarray}
  contradicting the fact that $\delta(a_0)$ is not a zero
  divisor. Hence $\frac{\delta(a_{0})}{\hat{g}}$ is not a zero
  divisor. It follows that
  \begin{eqnarray*}
    f\frac{\delta(a)}{\hat{g}} = 0,
    \quad\forall\, a \in \A \implies f\frac{\delta(a_{0})}{\hat{g}} = 0 \implies f = 0,	
  \end{eqnarray*}
  showing that $\frac{\delta}{\hat{g}}$ is indeed a basis for $\TSigma_{\hat{g}}.$
\end{proof}

\noindent
Next, let us explore the possibility of defining a $\sigmatau$-Lie
algebra structure on $\TSigma_{\hat{g}}$.

\begin{proposition}\label{prop sigmatau Lie alg}
  Let $\sigma\in\Aut(\A)$ and assume that $(\sigma,\Id)$ is a regular
  pair of endomorphisms. If
  $R:\TSigma_{\hat{g}}\otimes_{\complex} \TSigma_{\hat{g}}\to
  \TSigma_{\hat{g}} \otimes_{\complex} \TSigma_{\hat{g}}$ is defined
  by
  \begin{eqnarray*}
    R(X_{f} \otimes_{\complex} X_{g}) = X_{\sigma(g)}\otimes_{\complex} X_{\sigma^{-1}(f)}
  \end{eqnarray*}
  for $X_{f}, X_{g} \in \TSigma_{\hat{g}}$ then
  $(\TSigma_{\hat{g}}, R)$ is a $\sigmatau$-Lie algebra.
\end{proposition}

\begin{proof}
  Let
  $m:\TSigma_{\hat{g}} \otimes_{\complex}\TSigma_{\hat{g}} \to
  \TSigma_{\hat{g}}$ denote the composition
  $m(X_{f}\otimes_{\complex}X_{g}) = X_{f}\circ X_{g}.$ To prove that
  $(\TSigma_{\hat{g}}, R)$ is a $\sigmatau$-Lie algebra, one needs to
  show that
  $R^{2} =
  \mbox{Id}_{\TSigma_{\hat{g}}\otimes_{\complex}\TSigma_{\hat{g}}}$
  and
  \begin{eqnarray}\label{mult TSigma}
    m(X_{f}\otimesc X_{g} - R(X_{f}\otimesc X_{g})) \in \TSigma_{\hat{g}}.
  \end{eqnarray}
  Let us first prove that
  $R^{2} = \mbox{Id}_{\TSigma_{\hat{g}}\otimesc \TSigma_{\hat{g}}}.$
  For $X_{f}, X_{g} \in \TSigma_{\hat{g}},$ one has
  \begin{eqnarray*}
    R^{2}(X_{f}\otimesc X_{g}) = R(X_{\sigma(g)}\otimesc X_{\sigma^{-1}(f)}) = X_{\sigma(\sigma^{-1}(f))}\otimesc X_{\sigma^{-1}(\sigma(g))} = X_{f}\otimesc X_{g}.
  \end{eqnarray*}
  To prove \eqref{mult TSigma}, note that 
  \begin{eqnarray*}
    m(X_{f}\otimesc X_{g})(a) = X_{f}\circ X_{g}(a) = fX_{\mathds{1}}(gX_{\mathds{1}}(a))
  \end{eqnarray*}
  for $a \in \A$.  One has
  \begin{eqnarray*}
    &&m(X_{f}\otimesc X_{g} - R(X_{f} \otimesc X_{g}))(a) = X_{f}\circ X_{g}(a) - X_{\sigma(g)}\circ X_{\sigma^{-1}(f)}(a)\\
    && = fX_{\mathds{1}}(gX_{\mathds{1}}(a)) - \sigma(g)X_{\mathds{1}}(\sigma^{-1}(f)X_{\mathds{1}}(a)) = f\sigma(g)X_{\mathds{1}}(X_{\mathds{1}}(a)) + fX_{\mathds{1}}(g)X_{\mathds{1}}(a)\\
    &&- \sigma(g)\sigma(\sigma^{-1}(f))X_{\mathds{1}}(X_{\mathds{1}}(a)) - \sigma(g)X_{\mathds{1}}(\sigma^{-1}(f))X_{\mathds{1}}(a)\\
    && = fX_{\mathds{1}}(g)X_{\mathds{1}}(a) - \sigma(g)X_{\mathds{1}}(\sigma^{-1}(f))X_{\mathds{1}}(a) \\
    &&= (X_{f}(g) - X_{\sigma(g)}(\sigma^{-1}(f)))X_{\mathds{1}}(a)
       =X_{X_{f}(g) - X_{\sigma(g)}(\sigma^{-1}(f))}(a)
  \end{eqnarray*}
  implying that
  $ m(X_{f}\otimesc X_{g} - R(X_{f}\otimesc X_{g})) \in
  \TSigma_{\hat{g}}.$ Thus, $(\TSigma_{\hat{g}}, R)$ is a
  $\sigmatau$-Lie algebra.
\end{proof}

\noindent
The bracket on the $\sigmatau$-Lie algebra $(\TSigma_{\hat{g}}, R)$ in
the proposition above is given by
\begin{eqnarray*}
  [X_{f}, X_{g}]_{R} = \paraa{X_{f}(g) - X_{\sigma(g)}(\sigma^{-1}(f))} X_{\mathds{1}}
\end{eqnarray*}
for $X_{f}, X_{g} \in \TSigma_{\hat{g}}$.

For arbitrary $\sigma,\tau\in\End(\A)$, $\TSigma_{\hat{g}}$ can be
  given the structure of a $\Sigma_{\hat{g}}$-module. Namely, defining
  $\sigmah,\tauh:\TSigma_{\hat{g}}\to\TSigma_{\hat{g}}$ by
\begin{align*}
  \sigmah(X_f) = X_{\sigma(f)}\quad\text{ and }\quad\tauh(X_f) = X_{\tau(f)}
\end{align*}
one finds that
\begin{align*}
  \sigmah(fX_g) = \sigmah(X_{fg}) = X_{\sigma(f)\sigma(g)}
  =\sigma(f)X_{\sigma(g)} = \sigma(f)\sigmah(X_g)
\end{align*}
showing (together with obvious linearity properties and a similar
computation for $\tauh$) that $(\TSigma,\{\sigmah,\tauh\})$ is a
$\Sigma_{\hat{g}}$-module. Let us now study the curvature of a
connection on the $\Sigma_{\hat{g}}$-module
$(\TSigma,\{\sigmah,\tauh\})$ where $\sigmatau$ is assumed to be a
regular pair of endomorphisms, implying that $\TSigma_{\hat{g}}$ is a
free module. Define
$\nabla: \TSigma_{\hat{g}} \times
\TSigma_{\hat{g}}\to\TSigma_{\hat{g}}$ by
\begin{equation}\label{eq:comm.conn.unique}
  \nabla_{X_{f}}(X_{g}) = fX_{\delta(g)}=f\paraa{\tauh(X_g)-\sigmah(X_g)}
\end{equation}
for $X_{f}, X_{g} \in \TSigma_{\hat{g}}$, and note that for $f,g,h\in\A$
\begin{align*}
  \nabla_{X_f}(gX_{h})
  &=\nabla_{X_f}(X_{gh})
  = fX_{\delta(gh)}
    = fX_{\sigma(g)\delta(h)+\delta(g)\tau(h)}\\
  &= f\sigma(g)X_{\delta(h)}+f\delta(g)X_{\tau(h)}
    = \sigma(g)\nabla_{X_f}(X_h)+X_f(g)\tauh(X_h)\\
  \nabla_{X_f}(gX_{h})
  &=\nabla_{X_f}(X_{gh})
  = fX_{\delta(gh)}
    = fX_{\tau(g)\delta(h)+\delta(g)\sigma(h)}\\
  &= f\tau(g)X_{\delta(h)}+f\delta(g)X_{\sigma(h)}
    = \tau(g)\nabla_{X_f}(X_h)+X_f(g)\sigmah(X_h)  
\end{align*}
showing that $\nabla$ is a symmetric $\sigmatau$-connection. Moreover,
it holds that
\begin{align*}
  \nabla_{fX_{g}}(X_{h})
  =\nabla_{X_{fg}}(X_{h})=fgX_{\delta(h)} = f\nabla_{X_{g}}(X_{h})
\end{align*}
for $f,g,h\in\A$. Note that if $\sigmatau$ is a strongly regular pair,
then the connection defined in \eqref{eq:comm.conn.unique} is the
unique symmetric $\sigmatau$-connection on $\TSigma_{\hat{g}}$
(cf. Corollary~\ref{cor:unique.symmetric.connection}). Let us now
compute the curvature of $\nabla$ in the setting of
Proposition~\ref{prop sigmatau Lie alg}.

\begin{proposition}\label{prop:nabla.linear.curvature.zero}
  Assume that $\Sigma_{\hat{g}} = (\A, M_{\hat{g}})$ is a regular $\sigmatau$-algebra such that
  $(\TSigma_{\hat{g}}, R)$ is a $(\sigma, \mbox{Id})$-Lie algebra with
  $R$ given by
  \begin{eqnarray*}
    R(X_{f}\otimes_{\complex} X_{g}) = X_{\sigma(g)}\otimes_{\complex} X_{\sigma^{-1}(f)}
  \end{eqnarray*}
  for $X_{f}, X_{g} \in \TSigma_{\hat{g}}$ and let
  $(M, \{(\sigmah, \tauh)\})$ be a (left)
  $\Sigma_{\hat{g}}$-module. If
  $\nabla : \TSigma_{\hat{g}} \times M\to M$ is a left
  $(\sigma, \mbox{Id})$-connection such that
  $\nabla_{fX_{g}}(m) = f\nabla_{X_{g}}(m)$ for all $f, g \in \A,$
  then the curvature of $\nabla$ is zero.
\end{proposition}

\begin{proof}
  For $f,g\in\A$ one computes
\begin{align*}
\operatorname{Curv}(X_{f}, X_{g})(m) &= \nabla_{X_{f}}(\nabla_{X_{g}}(m)) - \nabla_{X_{\sigma(g)}}(\nabla_{X_{\sigma^{-1}(f)}}(m)) - \nabla_{[X_{f}, X_{g}]_{R}}(m)\\
&= \nabla_{X_{f}}(g\nabla_{X_{\mathds{1}}}(m)) - \nabla_{X_{\sigma(g)}}(\sigma^{-1}(f)\nabla_{X_{\mathds{1}}}(m)) - \nabla_{[X_{f}, X_{g}]_{R}}(m)\\
& = \sigma(g)\nabla_{X_{f}}(\nabla_{X_{\mathds{1}}}(m)) + X_{f}(g)\nabla_{X_{1}}(m) - \sigma(\sigma^{-1}(f))\nabla_{X_{\sigma(g)}}(\nabla_{X_{\mathds{1}}}(m))\\
& \qquad- X_{\sigma(g)}(\sigma^{-1}(f))\nabla_{X_{\mathds{1}}}(m) - \nabla_{[X_{f}, X_{g}]_{R}}(m)\\
&=  \sigma(g)f\nabla_{X_{\mathds{1}}}(\nabla_{X_{\mathds{1}}}(m)) + fX_{\mathds{1}}(g)\nabla_{X_{1}}(m) - f\sigma(g)\nabla_{X_{\mathds{1}}}(\nabla_{X_{\mathds{1}}}(m))\\
& \qquad- \sigma(g)X_{\mathds{1}}(\sigma^{-1}(f))\nabla_{X_{\mathds{1}}}(m) 
 - \nabla_{[X_{f}, X_{g}]_{R}}(m)\\
 &= (fX_{\mathds{1}}(g) - \sigma(g)X_{\mathds{1}}(\sigma^{-1}(f)))\nabla_{X_{\mathds{1}}}(m) -  \nabla_{[X_{f}, X_{g}]_{R}}(m)\\
 & = \nabla_{[X_{f}, X_{g}]_{R}}(m) - \nabla_{[X_{f}, X_{g}]_{R}}(m)= 0,
\end{align*}
showing that the curvature is indeed zero.
\end{proof}

\noindent
In particular, Proposition~\ref{prop:nabla.linear.curvature.zero}
implies that the curvature of the $\sigmatau$-connection on
$\TSigma_{\hat{g}}$ given in \eqref{eq:comm.conn.unique} is zero.  

We remark that not all connections on $\TSigma_{\hat{g}}$-module have
zero curvature. For example, the curvature of the connection defined
by $\nabla_{X_{f}}(X_{\mathds{1}}) = \gamma(f)\cdot X_{\mathds{1}}$
for $\mathbb{K}$-linear map $\gamma : \A\to\A,$ is
\begin{align*}
  \curv(&X_{f}, X_{g})X_{\mathds{1}}
    = \left(\sigma(\gamma(g))\gamma(f) - \sigma(\gamma(\sigma^{-1}(f)))\gamma(\sigma(g))\right)\cdot X_{\mathds{1}}\\
    &+\left(X_{f}(\gamma(g)) - X_{\sigma(g)}(\gamma(\sigma^{-1}(f)))\right)\cdot X_{\mathds{1}} - \gamma(X_{f}(g) - X_{\sigma(g)}(\sigma^{-1}(f)))\cdot X_{\mathds{1}}. 
\end{align*}

\bibliographystyle{alpha}
\bibliography{references}

\end{document}